\documentclass[a4paper,12pt,oneside,reqno]{amsart}

\usepackage[utf8]{inputenc}
\usepackage[T1]{fontenc}
\usepackage[english]{babel}
\usepackage{csquotes}
\usepackage{amsmath,amssymb,bm}	
\usepackage[all]{xy}
\usepackage{amsthm}
\usepackage[left=2.5cm, right=2.5cm, top=3cm, bottom=3cm]{geometry}
\usepackage[backend=bibtex,style=alphabetic,maxnames=5]{biblatex}
\usepackage[pdftex]{hyperref}
\usepackage{todonotes}
\usepackage{pigpen}
\usepackage{pict2e}	
\usepackage{theoremref}
\usepackage{mathrsfs}

\theoremstyle{plain}
\newtheorem{thm}{Theorem}[section]
\newtheorem{cor}[thm]{Corollary}
\newtheorem{prop}[thm]{Proposition}
\newtheorem{lem}[thm]{Lemma}

\theoremstyle{definition}
\newtheorem{defi}[thm]{Definition}

\theoremstyle{remark}
\newtheorem{rmk}[thm]{Remark}

\newtheorem{exam}[thm]{Example}

\def\dhMakeCatname#1{\expandafter\ifx\csname #1\endcsname\relax \expandafter\def\csname #1\endcsname{\textbf{\upshape #1}}%
	\else\expandafter\def\csname cat#1\endcsname{\textbf{\upshape #1}}\fi}
\def\dhMakeCal#1{\expandafter\ifx\csname c#1\endcsname\relax \expandafter\def\csname c#1\endcsname{\mathcal{#1}}%
	\else\expandafter\def\csname dhc#1\endcsname{\mathcal{#1}}\fi}
\def\dhMakeBB#1{\expandafter\ifx\csname #1\endcsname\relax \expandafter\def\csname #1\endcsname{\mathbb{#1}}%
	\else\expandafter\def\csname dhbb#1\endcsname{\mathbb{#1}}\fi}

\newcommand\menge[2]{\left\{#1\ \middle|\vphantom{\big|}\ #2\right\}}

\newcommand\Fun{\mathrm{Fun}}
\newcommand\Map{\mathrm{Map}}
\newcommand\BC{\textbf{\upshape{BornCoarse}}}
\newcommand\GBC{\ensuremath{\bm{\Gamma}}\textbf{\upshape{BornCoarse}}}
\newcommand\GSet{\ensuremath{\bm{\Gamma}}\textbf{\upshape{Set}}}
\newcommand\GgBC{\ensuremath{\bm{\Gamma}}\textbf{\upshape B$\widetilde{\textbf{\upshape orn}}$Coarse}}
\newcommand\GgSpc{\ensuremath{\bm{\Gamma}}\textbf{\upshape S$\widetilde{\textbf{\upshape pc}}\bm{\mathcal{X}}$}}
\newcommand\GgSp{\ensuremath{\bm{\Gamma}}\textbf{\upshape S$\widetilde{\textbf{\upshape p}}\bm{\mathcal{X}}$}}
\newcommand\GSp{\ensuremath{\bm{\Gamma}}\textbf{\upshape S${\textbf{\upshape p}}\bm{\mathcal{X}}$}}

\newcommand\gBC{\textbf{\upshape B$\widetilde{\textbf{\upshape orn}}$Coarse}}
\newcommand\PShSp{\PSh_{\Sp}}
\newcommand\ShSp{\Sh_{\Sp}^{\tau}}
\newcommand{\ShSpMot}{\Sh_{\Sp}^{\tau,\mathrm{mot}}}
\dhMakeCatname{Set}
\dhMakeCatname{Sh}
\dhMakeCatname{Sp}
\dhMakeCatname{Spc}
\dhMakeCatname{PSh}

\dhMakeCal{B}
\dhMakeCal{C}
\dhMakeCal{G}
\dhMakeCal{E}
\dhMakeCal{D}
\dhMakeCal{P}
\dhMakeCal{S}
\dhMakeCal{T}
\dhMakeCal{F}
\dhMakeCal{Y}
\dhMakeCal{L}

\newcommand\sS{\mathscr S}

\dhMakeBB{N}
\dhMakeBB{R}
\dhMakeBB{Z}
\dhMakeBB{I}

\let\odmin\min
\let\odmax\max
\renewcommand{\min}{{\begingroup\odmin\endgroup}}
\renewcommand{\max}{{\begingroup\odmax\endgroup}}

\let\leer\emptyset
\newcommand\bI{\begingroup \bm{\mathrm{I}}\endgroup}
\DeclareMathOperator*{\colim}{colim}

\newcommand\coarse{{\begingroup\mathrm{coarse}\endgroup}}

\newcommand\gYo{\mathrm{Y\tilde o}}

\newcommand\iso{\overset{{\lower.5em\hbox{$\scriptstyle\cong$}}}{\longrightarrow}}

\makeatletter
\newcommand{\adjunction}{\@ifstar\named@adjunction\normal@adjunction}
\newcommand{\normal@adjunction}[4]{%
	#1\colon #2%
	\mathrel{\vcenter{%
			\offinterlineskip\m@th
			\ialign{%
				\hfil$##$\hfil\cr
				\longrightarrow\cr  
				\noalign{\kern-.3ex}
				\smallbot\cr
				\longleftarrow\cr
			}%
	}}%
	#3 \noloc #4%
}
\newcommand{\named@adjunction}[4]{%
	#2%
	\mathrel{\vcenter{%
			\offinterlineskip\m@th
			\ialign{%
				\hfil$##$\hfil\cr
				\scriptstyle#1\cr
				\noalign{\kern.1ex}
				\longrightarrow\cr
				\noalign{\kern-.3ex}
				\smallbot\cr
				\longleftarrow\cr
				\scriptstyle#4\cr
			}%
	}}%
	#3%
}
\newcommand\noloc{%
	\nobreak
	\mspace{6mu plus 1mu}
	{:}
	\nonscript\mkern-\thinmuskip
	\mathpunct{}
	\mspace{2mu}
}
\newcommand{\smallbot}{%
	\begingroup\setlength\unitlength{.15em}%
	\begin{picture}(1,1)
	\roundcap
	\polyline(0,0)(1,0)
	\polyline(0.5,0)(0.5,1)
	\end{picture}%
	\endgroup
}
\makeatother

\numberwithin{equation}{section}

\hypersetup{
	linktocpage = true,
	colorlinks,
	linkcolor=[RGB]{0,0,120},
	citecolor=[RGB]{0,0,120},
	urlcolor=[RGB]{0,0,120}
}

\makeatletter
\def\pdfdaten{
	\hypersetup{
		pdftitle = {\@title},
		pdfauthor = {\@author},
		pdfkeywords = {\@keywords},    
		bookmarksopen = true,
		bookmarksopenlevel = 1
}}  
\makeatother

\makeatletter
\def\subsection{\@startsection{subsection}{2}%
	\z@{2.5\linespacing\@plus.7\linespacing}{1.0\linespacing}%
	{\normalfont\scshape}}
\makeatother

\title[Generalized coarse motivic spectra]{Generalized bornological coarse spaces and coarse motivic spectra}
\author{Daniel Hei\ss}
\address{Fakultät für Mathematik, Universität Regensburg, Germany}
\thanks{\textsc{Fakultät für Mathematik, Universität Regensburg, Germany}}
\date{\today}

\begin{document}
		\begin{abstract}
		We generalize the notion of a bornology by omitting the condition that a one-point-subset is bounded and obtain a complete and co-complete generalization of the category of
		bornological coarse spaces. Then we imitate the construction of motivic coarse spectra in this new setting and show that the inclusion functor from the category of 
		bornological coarse spaces to its generalization induces an equivalence of motivic coarse spectra. In particular, for any stable co-complete $\infty$-category $C$, it induces
		an equivalence between the category of $C$-valued coarse homology theories on bornological coarse spaces and the category of $C$-valued coarse homology theories on
		generalized bornological coarse spaces.
	\end{abstract}
	
	\maketitle
		
	\tableofcontents
	
	\section{Introduction}
		In \cite{Uli} the authors introduce the category \BC~of bornological coarse spaces together with proper and controlled morphisms. Based on this, they define the notion
		of coarse homology theories and construct the universal coarse homology theory $\BC\to\Sp\bm{\mathcal{X}}$, where the target is 
		the stable $\infty$-category of motivic coarse spectra $\Sp\bm{\mathcal X}$. This situation is generalized in \cite{equ.Uli} to \GBC, the category of
		$\Gamma$-equivariant bornological coarse spaces and accordingly equivariant coarse homology theories, equivariant coarse motivic spectra $\GSp$ and the universal equivariant
		coarse homology theory $\GBC\to\GSp$.
		
		 The basic categories \BC~and \GBC~both lack some nice categorical properties, for example they do not have a 
		final object and there are various diagrams that do not admit a pushout. Both is shown in section \ref{incomplete.sec} of this paper.

		The structural reason for the non-existence in both situations is
		the fact that every single point is bounded by definitions; and -- in fact -- if we dropped that requirement, there exists a final object and the pushout in \thref{exa.N}.
		
		This motivates the following questions: \begin{enumerate}
			\item If we drop the condition that every point is bounded, do we get a reasonable notion of generalized bornological coarse spaces?
			\item Is this new category complete and co-complete?
			\item What are the relations between the new category and the (classical) category \BC?
			\item How does the theory of motivic coarse spaces and spectra over the new category fit into the old framework?
		\end{enumerate} \vspace{1cm}
		
		This paper is organized as follows: \ In order to provide the reader with a mostly self-contained paper, we recollect the basic framework of \BC~in section 2. We then briefly 
		illustrate the non-completeness and non-co-completeness of \BC~in section 3. After that, we construct the new categories \gBC~and \GgBC. Basically we simply drop the condition
		\enquote{every one-point-set is bounded} from the axioms of a bornology. Further we elaborate basic notions in that new framework.\\
		In section 5 we
		prove the completeness and co-completeness of these new categories and then we briefly study the connections of \GBC~and \GgBC~in section 6. We will see, that although the inclusion functor
		$U\colon\GBC\to\GgBC$  is neither left nor right adjoint, it does preserve all limits and colimits. Further we elaborate a criterion for which (co-)limits exist in \GBC.\\
		Finally in section 7 we define generalized motivic coarse spaces and generalized motivic coarse spectra $\GgSp$ and we prove that the fully faithful inclusion 
		functor $U$ induces an equivalence
		of the $\infty$-categories $\GSp$ and $\GgSp$. In particular, for any
		co-complete stable $\infty$-category $\cC$ it induces  an
		equivalence of $\cC$-valued coarse homology theories on $\GBC$ and those on $\GgBC$.

\clearpage
\section{The category \BC}
	In order to provide the reader with a mostly self-contained paper, in this section we recall the category \BC~of bornological coarse spaces introduced by \cite{Uli}.

For a set $X$ we denote by $\cP(X)$ the powerset of $X$. Further consider two subsets $U$ and $V$ of $X\times X$ and a subset $B$ of $X$, then we define
\begin{align*}
	U^{-1}&:=\menge{(x,y)\in X\times X}{(y,x)\in U}\tag{Inverses}\\
	U\circ V&:=\menge{(x,y)\in X\times X}{\exists z\in X:\ (x,z)\in U,\,(z,y)\in V}\tag{Compositions}\\
	U[B]&:=\menge{x\in X}{\exists b\in B:\ (x,b)\in U}\tag{Thickenings}
\end{align*}

Let $\cB_X$ be a subset of $\cP(X)$.
\begin{defi}
	The set $\cB_X$ is called a \emph{bornology} on $X$, if it contains all finite subsets of $X$ and is closed under taking subsets and finite unions. We refer to the elements in $\cB_X$ as 
	\emph{bounded subsets} of $X$.
\end{defi}

Consider a subset $\cC_X$ of $\cP(X\times X)$.
\begin{defi}
	The set $\cC_X$ is a \emph{coarse structure} on $X$, if it contains the diagonal $\Delta_X$ of $X$ and is closed under taking subsets, finite unions, inverses and compositions. We refer to
	elements of $\cC_X$ as 
	\emph{(coarse) entourages} of $X$ or as \emph{controlled sets}.
\end{defi}

Assume that $\cB_X$ is a bornology on $X$ and that $\cC_X$ is a coarse structure on $X$.

\begin{defi}
	We say that $\cB_X$ and $\cC_X$ are \emph{compatible}, if for all bounded subsets $B$ of $X$ and all entourages $U$ in $\cC_X$ the $U$-thickening $U[B]$ is again a bounded 
	subset of $X$.
\end{defi}

Let $Y$ be another set together with a bornology $\cB_Y$ and a coarse structure $\cC_Y$. Consider a map $f\colon X\to Y$.
\begin{defi}
	The map $f$ is called \emph{proper}, if for all bounded subsets $B$ of $Y$ the pre-image $f^{-1}(B)$ is a bounded subset of $X$. Moreover $f$ is \emph{controlled}, if
	for all entourages $U$ in $\cC_X$ we have $(f\times f)(U)\in\cC_Y$.
\end{defi}

\begin{defi}
	A \emph{bornological coarse space} is a set $X$ together with a bornology $\cB_X$ on $X$ and a coarse structure $\cC_X$ on $X$ that is compatible with $\cB_X$.
	A morphism between two bornological coarse spaces $X$ and $Y$ is a map $f\colon X\to Y$, which is proper and controlled.
	The category of bornological coarse spaces is denoted by \BC.
\end{defi}

For subsets $\cS$ of $\cP(X)$ and $\cT$ of $\cP(X\times X)$ we denote by $\cB\langle\cS\rangle$ the smallest bornology on $X$, which contains the elements of $\cS$ and by
$\cC\langle\cT\rangle$ the smallest coarse structure on $X$ containing all elements of $\cT$. They are called the bornology resp. the coarse structure \emph{generated} by
$\cS$ resp. $\cT$.

\begin{exam}
	\begin{enumerate}
		\item The bornology $\cB_\min:=\cB\langle\leer\rangle$ generated by the empty set is called the \emph{minimal bornology}. It consists of all finite subsets of $X$.
		\item The coarse structure $\cC_\min:=\cC\langle\leer\rangle$ generated by the empty set is called the \emph{minimal coarse structure}. It contains precisely the subsets of the diagonal $\Delta_X$.
		\item The powerset $\cP(X)$ is a bornology on $X$, it is called the \emph{maximal bornology} $\cB_\max$. Likewise $\cP(X\times X)$ is a coarse structure on $X$, the \emph{maximal
			coarse structure} $\cC_\max$.
	\end{enumerate}
\end{exam}

\begin{rmk}\thlabel{comp.max.min.class}
	The maximal bornology is compatible with all coarse structures, in particular with the maximal one. In fact it is the only bornology compatible with $\cC_{\max}$. \ The minimal coarse
	structure is compatible with all bornologies, in particular with the minimal bornology.
\end{rmk}

Motivated by the  remark above we denote by $X_{\min,\min}$ the bornological coarse space $X$ equipped with the minimal bornology and coarse structure. Further by
$X_{\min,\max}$ we denote the space $X$ equipped with the minimal coarse and the maximal bornological structure. Finally, if we equip $X$ with the maximal bornological and coarse structure, we
denote the space by $X_{\max,\max}$.

\begin{lem}\thlabel{forget.adjoint}
	The above constructions extend to fully faithful functors \[(-)_{\min,\min} ,\ (-)_{\min,\max},\ (-)_{\max,\max}\colon\Set\longrightarrow\BC.\]
	Moreover we have an adjunction \[\adjunction{(-)_{\min,\max}}{\Set}{\BC}{U}\]
	where $U$ denotes the forgetful functor.
\end{lem}
\begin{proof}
	Clear from the definitions.
\end{proof}

Bornologies and coarse structures are often defined by declaring a generating set. In order to check that the generated bornological and coarse structures are compatible, it is enough to 
verify the compatibility condition for the
generators. The following lemma makes this precise:

\begin{lem}\thlabel{class.check.gen}
	Assume that for all $U$ in $\cT$ we have $U^{-1}\in\cT$ and that for all $B$ in $\cS$ and all $U$ in $\cT$
	the $U$-thickening $U[B]$ is contained in a finite union of sets in $\cS$, then the structures $\cB\langle\cS\rangle$ and $\cC\langle\cT\rangle$ are compatible.
\end{lem}
\begin{proof}
	Easy calculations.
\end{proof}

Similar statements can be formulated for verifying, whether a given map is proper or controlled. In more detail: Assume $f\colon X\to Y$ is a map of sets both equipped with a bornological coarse
structure.

\begin{lem}\thlabel{class.map.gen}
\begin{enumerate}\item
	If there is a subset $\cS$ of $\cP(Y)$ such that $\cB_Y=\cB\langle\cS\rangle$ and if for every $B$ in $\cS$ we have $f^{-1}(B)\in\cB_X$, then $f$ is proper. \item If a subset $\cT$ of 
	$\cP(X\times X)$ generates the coarse structure $\cC_X$ on $X$, then $f$ is controlled iff for all $U$ in $\cT$ we have $(f\times f)(U)\in\cC_Y$.
\end{enumerate}
\end{lem}
\begin{proof}
	Elementary calculations.
\end{proof}

Consider a map of sets $f\colon Y\to X$, where $X$ is equipped with a bornological coarse structure. Then we define the \emph{pullbacks} of these structures as
\[f^*\cC_X:=\cC\left\langle\menge{(f\times f)^{-1}(U)}{U\in\cC_X}\right\rangle,\qquad f^*\cB_X:=\cB\left\langle\menge{f^{-1}(B)}{B\in\cB_X}\right\rangle.\]

\begin{lem}\thlabel{class.comp}
	The structures $f^*\cC_X$ and $f^*\cB_X$ are compatible and define a bornological coarse structure on $Y$. Furthermore the map $f$ becomes a morphism
	$Y\to X$ of bornological coarse spaces.
\end{lem}
\begin{proof}
	Compatibility can be checked on generators by \thref{class.check.gen} (note that the set of generators of $f^*\cC_X$ is closed under taking inverses). It is immediate, that
	\[\left((f\times f)^{-1}(U)\left)\left[f^{-1}(B)\right]\right.\right.\subseteq f^{-1}\left(U[B]\right)\] and thus the compatibility. That $f$ defines a morphism in \BC~follows immediately from the definitions and 
	\thref{class.map.gen}.
\end{proof}

Let $\Gamma$ be a group.

\begin{defi}\thlabel{def.equiv.bc}
	A bornological coarse space $X$ is said to be \emph{$\Gamma$-equivariant} or shortly a \emph{$\Gamma$-bornological coarse space}, 
	if $\Gamma$ acts on $X$ by automorphisms such that the set $\cC_X^\Gamma$ of $\Gamma$-invariant
	entourages is cofinal in $\cC_X$. \ The subcategory of \BC~consisting of $\Gamma$-equivariant spaces together with $\Gamma$-equivariant morphisms is denoted by \GBC.
\end{defi}

\begin{rmk}
	The set $\cC_X^\Gamma$ of $\Gamma$-invariant entourages is cofinal iff every entourage $U$ in $\cC_X$ is contained in a $\Gamma$-invariant one.\\
	If $\cT$ is a collection of $\Gamma$-invariant subsets of $X\times X$, then $X$ equipped with $\cC\langle\cT\rangle$ and a compatible bornological structure is a $\Gamma$-equivariant 
	bornological coarse space.
\end{rmk}

At this point we elaborate an instructive example:
\begin{exam}
	Let $(X,d)$ be a metric space. We define the \emph{bornological coarse structure induced by the metric}: 
	\[\cB_d:=\cB\left\langle\menge{B_r(x)}{x\in X,\, r\geq0}\right\rangle,\qquad \cC_d:=\cC\left\langle\menge{U_r}{r\geq0}\right\rangle,\]
	where \[U_r:=\menge{(x,y)\in X}{d(x,y)<r}.\]
	Let further $\Gamma$ be a group acting isometrically on $X$.\\
	The coarse and the bornological structures are compatible because $U_r[B_s(x)]\subseteq B_{r+s}(x)$ and the generators of $\cC_d$ are obviously $\Gamma$-invariant, hence
	$X$ together with $\cC_d$ and $\cB_d$ form a $\Gamma$-equivariant bornological coarse space. 
\end{exam}

	Consider two $\Gamma$-equivariant bornological coarse spaces $X$ and $Y$. We define their \emph{tensor product} $X\otimes Y$ as follows: The underlying set of $X\otimes Y$
	is the cartesian product $X\times Y$. It is equipped with the bornological coarse structure
	\begin{align}
		\cC_{X\otimes Y}&:=\cC\left\langle\menge{U\times U'}{U\in\cC_X,\,U'\in\cC_Y}\right\rangle,\nonumber\\
		\cB_{X\otimes Y}&:=\cB\left\langle\menge{B\times B'}{B\in\cB_X,\,B'\in\cB_Y}\right\rangle.\label{class.tensor}
		\end{align}

\begin{lem}\thlabel{tensor-symm}
	The tensor product extends to a bi-functor
	\[-\otimes-\colon\GBC\times\GBC\longrightarrow\GBC\]
	which is part of a symmetric monoidal structure on \GBC~with tensor unit the one-point-space $\ast$ equipped with $\cC_\max$ and $\cB_\max$.
\end{lem}
\begin{proof}
	See \cite[Example 2.18]{equ.Uli}.
\end{proof}

\begin{rmk}
	The category \BC~has all non-empty products (\cite[Lemma 2.22]{Uli}), but it is not symmetric monoidal w.r.t. the Cartesian product, 
	because there is no unit:
	In fact, by the universal property of the product, a unit is final and such an element cannot exist in $\BC$ as shown in \thref{no.final} below.
\end{rmk}

\clearpage
\section{Incompleteness of \BC}\label{incomplete.sec}
	As mentioned in the introduction the category of bornological coarse spaces is neither complete nor co-complete. In this section we show that \BC~is lacking a final object and give an
	instructive concrete example for non-existing pushouts. Both defects arise from the fact, that a subset containing just one element is bounded by definition of a bornology. This  
	assumption will be the only one we drop in order to define generalized bornological coarse spaces in section \ref{sec.gBC}, giving a bi-complete category.
	
	\begin{lem}\thlabel{no.final}
		The category \,\BC~does not have a final object.
	\end{lem}
	\begin{proof}
		The forgetful functor $U\colon\BC\to\Set$ is right adjoint to the functor $(-)_{\min,\max}$ (cf. \thref{forget.adjoint}), hence it preserves limits. So if there was a final object in \BC, the underlying set would have to be a singleton $\ast$. However assume $\ast$ is final in \BC, then for any
		$X$ in $\BC$ there exists a morphism $X\to\ast$, which in particular is proper, hence $X\in\cB_X$, which is false in general.
	\end{proof}
	
	In fact the final object is the only non-existing limit in \BC~(see \thref{cl.almost.compl} or \cite{Uli}). However there are various non-existing colimits. For a nice example we consider the following
	pushout:
	\begin{exam}\thlabel{exa.N}
		The following diagram in \BC~does not admit a pushout:\[
		\xymatrix{\N_{\min,\max}\ar[r]\ar[d] & \N_{\max,\max}\\ \N_{\min,\min}}
		\] where both maps are the identity on the underlying sets. Moreover, there cannot even exist a bornological coarse space $T$ together with morphisms $\alpha,\beta$ such that the diagram
		\[\xymatrix{
			\N_{\min,\max}\ar[r]\ar[d] & \N_{\max,\max}\ar[d]^\alpha\\
			\N_{\min,\min}\ar[r]_\beta & T}\tag{\dag}\label{NPushout}
		\]
		commutes. \ In fact, assume that there is a commutative square as in (\ref{NPushout}). Let $t:=\alpha(0)$. Then the subset $B:=\{t\}$ is bounded. Furthermore $\N\times\N$ is a controlled
		set in $\N_{\max,\max}$, hence $U:=(\alpha\times\alpha)(\N\times\N)\in\cC_T$. So, by compatibility of the bornological and coarse structure, the $U$-thickening $U[B]$ of $B$ is
		bounded in $T$. But then by properness of $\beta$, the pre-image $\beta^{-1}(U[B])$ is bounded in $\N_{\min,\min}$, hence finite. However since (\ref{NPushout}) commutes, the maps
		$\alpha$ and $\beta$ coincide, thus we get 
		$\beta^{-1}(U[B])=\N$, \ because for any $n$ in $\N$ we have $\beta(n)=\alpha(n)\in U[B]=U[\{\alpha(0)\}]$ (since $(\alpha(n),\alpha(0))\in U$). \ Thus the pre-image is not finite, a 
		contradiction.
	\end{exam}

In order to motivate the definition of generalized bornology in the next section, we point out that, if we drop the condition that every point is bounded, the one-point-space consisting of a
hypothetical unbounded point is clearly final. Furthermore the pushout in \thref{exa.N} would exist:

\begin{exam}\thlabel{ex.PO}
	Endow the natural numbers $\N$ with the maximal coarse structure and a \enquote{bornology} $\cB_\N:=\{\leer\}$, i.e. the only bounded subset is the empty-set. Denote this \enquote{space}
	by $\N_{\max,\leer}$. Then we obtain a pushout square
	\[\xymatrix{\N_{\min,\max}\ar[r]\ar[d] \ar@{}[dr]|(0.65){\text{\pigpenfont{I}}} & \N_{\max,\max}\ar[d]\\
		\N_{\min,\min}\ar[r] & \N_{\max,\leer}}\]
	where all the underlying maps are the identity. \\
	Indeed let $T$ be a test object and $\gamma,\delta$ be morphisms as in the following diagram:
	\[\xymatrix{\N_{\min,\max}\ar[r]\ar[d] & \N_{\max,\max} \ar[d]\ar@/^0.5cm/[rdd]^\gamma \\ \N_{\min,\min}\ar[r]\ar@/_0.5cm/[drr]_\delta & \N_{\max,\emptyset} \ar@{.>}[dr]^\phi\\ && T }\]
	Viewed as diagram in \Set~there is a unique map $\phi$, which is controlled, because $\gamma$ is controlled. Furthermore $\phi$ is proper, since the very same calculation as in
	\thref{exa.N} shows, that $\phi$ cannot hit any bounded point.
\end{exam}

\clearpage
\section{Construction of the categories \texorpdfstring{$\gBC$}{gBC} and \texorpdfstring{$\GgBC$}{GgBC}}\label{sec.gBC}
	The example above motivates the construction of generalized bornological coarse
	spaces, where we just remove the requirement, that every finite subset is bounded. We will see that we get a very similar theory, but that this new category \gBC~of generalized bornological
	coarse spaces has better categorical properties.

Let again $X$ be a set and let $\cB_X$ be a non-empty subset of $\cP(X)$.

\begin{defi}
	The set $\cB_X$ is called a \emph{generalized bornology} on $X$, if it closed under taking subsets\footnote{ In particular, the empty set is an element of $\cB_X$} and finite unions. 
	The elements of $\cB_X$ are called \emph{bounded subsets of $X$}.\\
	A \emph{generalized bornological space} is a set $X$ together with a generalized bornology $\cB_X$.
\end{defi}

Let $X$ and $Y$ be two generalized bornological spaces and $f\colon X\to Y$ be a map.
\begin{defi}
	The map $f$ is called \emph{proper}, if for all $B$ in $\cB_Y$ we have $f^{-1}(B)\in\cB_X$.
\end{defi}

\begin{exam}\mbox{} \begin{itemize}
		\item Every (classical) bornology on $X$ is a generalized bornology on $X$. Thus we adapt the definitions of the bornology induced by a metric on $X$, as well as the notions of the maximal and
		the minimal bornology (although the latter is no longer the \enquote{true} minimal one).
		\item The \emph{trivial bornology} $\cB_\emptyset:=\{\emptyset\}$ is a generalized bornology on $X$. We write $X_\leer$ for the generalized bornological space $X$ together with the trivial
		bornology. \\ For trivial reasons
		any map into $X_\emptyset$ is proper.
	\end{itemize}
\end{exam}

\begin{defi}\mbox{}\begin{itemize}
		\item The \emph{set of bounded points} of $X$ is defined as \[X_b:=\left\{x\in X\ \big|\ \{x\}\in\cB_X\right\}.\]
		Its complement $X_h:=X\setminus X_h$ is called the \emph{set of unbounded points} of $X$.
		\item The space $X$ is called \emph{locally bounded} if $X=X_b$.
		\item We call a subset $D$ of $X$ \emph{small} if $D\subseteq X_b$. Otherwise it is called \emph{big}.
	\end{itemize}
\end{defi}

\begin{rmk}
	We have $X_b=\bigcup_{B\in\cB_X}B$, hence $\cB_X$ is an (ordinary) bornology on $X$ iff $X$ is locally bounded.
\end{rmk}

Let $\cC_X$ be a coarse structure on $X$.

\begin{defi}
	We say, that $\cB_X$ and $\cC_X$ are \emph{compatible}, if for all bounded 
	sets $B$ in $\cB_X$ and all entourages $U$ in $\cC_X$, the $U$-thickening
	$U[B]$ is again a bounded subset of $X$.
\end{defi}

\begin{defi}
	A \emph{generalized bornological coarse space} is a  set $X$ together with a coarse structure $\cC_X$ and a generalized bornology on $\cB_X$, such that $\cB_X$ and $\cC_X$
	are compatible. \ A morphism $f\colon X\to Y$ between generalized bornological coarse spaces is a map $f\colon X\to Y$ that is proper and controlled.\\
	The category of generalized bornological coarse spaces is denoted by \gBC.
\end{defi}

\begin{exam}
	By trivial reasons the trivial bornology is compatible with every coarse structure. Hence $\cC_{\max}$ is not only compatible with the maximal but also with the trivial generalized bornology
	(cf. \thref{comp.max.min.class}).
\end{exam}

By this example we can endow $X$ with the maximal coarse structure $\cC_\max$ and the trivial bornology $\cB_\leer$ to get a generalized bornological coarse space, which we denote
by $X_{\max,\leer}$.

\begin{lem}\thlabel{adjoint}\mbox{}
	\begin{itemize}
		\item The construction above extends to a functor \[(-)_{\max,\leer}\colon\Set\longrightarrow\gBC\] which is part of an adjunction
		\[\adjunction{U}{\gBC}{\Set}{(-)_{\max,\leer}},\]
		where $U$ denotes the forgetful functor.
		\item Like in the classical setting we still have the adjunction 
		\[
			\adjunction{(-)_{\min,\max}}{\Set}{\gBC}{U}.
		\]
	\end{itemize}
\end{lem}
\begin{proof}
	Follows immediately from the definitions.
\end{proof}

Let $X$ be a generalized bornological coarse space and let $f\colon Y\to X$ be a map of sets. We can define the \emph{pullback}
\[f^*\cC_X:=\cC\left\langle\menge{(f\times f)^{-1}(U)}{U\in\cC_X}\right\rangle,\qquad f^*\cB_X:=\cB\left\langle\menge{f^{-1}(B)}{B\in\cB_X}\right\rangle.\]

\begin{lem}\thlabel{subspace.str}
	The structures $f^*\cC_X$ and $f^*\cB_X$ are compatible and endow $Y$ with the structure of a generalized bornological coarse space. Moreover the map $f$ becomes a morphism $Y\to X$
	in \gBC.
\end{lem}
\begin{proof}
	Same calculation as in \thref{class.comp}.
\end{proof}

Let $\Gamma$ be a group. 
Like in the setting of classical \BC~we can define the notion of $\Gamma$-equivariant generalized bornological coarse spaces. Since this affects only the coarse structure of the space we just
copy the definitions from \BC:

\begin{defi}
	A \emph{$\Gamma$-equivariant generalized bornological coarse space} is a generalized bornological coarse space 
	$X$ in $\gBC$ together with an action of $\Gamma$ on $X$ by automorphisms such that the set $\cC_X^\Gamma$ of $\Gamma$-invariant
	entourages is cofinal in $\cC_X$ (cf. \thref{def.equiv.bc} and the following remark). It is also refered to as \emph{generalized $\Gamma$-bornological coarse space}.\\
	The category of generalized $\Gamma$-bornological coarse spaces together with $\Gamma$-equivariant morphisms is denoted by \GgBC.
\end{defi} 

\begin{lem}\thlabel{gsubstr}
	Assume that in \thref{subspace.str} the space $X$ is in $\GgBC$ and that the map $f$ is $\Gamma$-equivariant, then the generalized bornological coarse space $Y$ defined
	by pullback is $\Gamma$-equivariant.
\end{lem}
\begin{proof}
	It is enough to verify that each generator $(f\times f)^{-1}(U)$ of $f^*\cC_X$ is $\Gamma$-invariant. W.l.o.g. 
	we may assume that $U$ is $\Gamma$-invariant since $\cC_X^\Gamma$ is cofinal in $\cC_X$. Now the claim follows from the $\Gamma$-equivariance of $f$.
\end{proof}

\clearpage
\section{Completeness and co-completeness of \texorpdfstring{\GgBC}{GgBC}}
	In this section we show that the categories \gBC~and \GgBC~are complete and co-complete. 
	Since the forgetful functor $U\colon\gBC\to\Set$ is left and right adjoint (cf.  \thref{adjoint}), it preserves limits and 
	colimits. Hence whenever we want to construct a (co-)limit, the underlying set
	is already determined (by the (co-)limit in \Set) and we just have to construct a suitable generalized bornological coarse structure.

\begin{prop}\thlabel{prod}
	The category $\gBC$ has all products.
\end{prop}
\begin{proof}
	Let $(X_i)_{i\in I}$ be a family of generalized bornological coarse spaces. The underlying set of the product will be the product of the underlying sets, so we define 
	\[
	X:=\prod_{i\in I}X_i,\qquad \cC_X:=\cC\left\langle\menge{\prod\nolimits_{i\in I}U_i}{\forall i\in I:\ U_i\in\cC_{X_i}}\right\rangle\]\[
	\cB_X:=\cB\left\langle\menge{B_j\times\prod\nolimits_{i\neq j}X_i}{j\in I,\ B_j\in\cB_{X_j}}\right\rangle.
	\]
	The canonical projections $\pi_i\colon X\to X_i$ are morphisms in \gBC~by construction. Now the space $X$ represents the product $\prod_{i\in I}X_i$ in
	\gBC: For any test-object $T$ in $\gBC$ with morphisms $f_i\colon T\to X_i$ there is a unique set-theoretical map $f\colon T\to X$. The properness of that map can be checked on generators,
	so let $B:=B_j\times\prod_{i\neq j}X_j$ be an arbitrary generator. Then $f^{-1}(B)\subseteq f_j^{-1}(\pi_j(B))=f^{-1}_j(B_j)$,
	where the latter set is bounded in $T$ since $f_j$ is proper, hence $f^{-1}(B)\in\cB_T$. \ 
	For an entourage $U$ in $\cC_T$, we have $(f_i\times f_i)(U)\in\cC_{X_i}$ for all $i$, hence $(f\times f)(U)=\prod_{i\in I}(f_i\times f_i)(U)\in\cC_X$.
\end{proof}

\begin{cor}\thlabel{gprod}
	The category $\GgBC$ has all products.
\end{cor}
\begin{proof}
	The proof of \thref{prod} generalizes to the $\Gamma$-equivariant case. The only thing left to show is the cofinality condition on the entourages:
	As $\cC_{X_i}^\Gamma$ are cofinal in $\cC_{X_i}$, we can restrict our generators for $\cC_X$ to be products of $\Gamma$-invariant $U_i$ in $\cC_{X_i}^\Gamma$.
	Moreover the diagonal $\Delta_X$ clearly is $\Gamma$-invariant, hence the claim follows.
\end{proof}

\begin{prop}\thlabel{tmp1}
	The category $\gBC$ has all equalizers of pairs of morphisms.
\end{prop}
\begin{proof}
	Let $f,g\colon X\to Y$ be two morphisms in \gBC. Define $E:=\operatorname{Eq}(f,g)$ to be the set-theoretical equalizer of the maps $f$ and $g$. Via the inclusion
	$\iota\colon E\to X$ in \Set~we can endow $E$ with a generalized bornological coarse structure $\iota^*\cC_X$ and $\iota^*\cB_X$ (cf. \thref{subspace.str}).
	 This object together with the canonical morphism
	into $X$ fulfills the universal property of the equalizer: \ Let $T$ be in $\gBC$ and let $h\colon T\to X$ be a morphism in \gBC~such that $f\circ h=g\circ h$. Then set-theoretically
	this map factors uniquely through $E$: 
	\[
		\xymatrix{E\ar[r]^\iota & X\ar@<1mm>[r]^f\ar@<-1mm>[r]_g & Y \\ T\ar@{.>}[u]^\ell \ar[ur]_h }
	\]
	Now $\ell$ is proper, because any generator $A$ of $\iota^*\cB_X$ is of the form $A=\iota^{-1}(B)$ for some $B$ in $\cB_X$, hence $\ell^{-1}(A)=h^{-1}(B)$, which is bounded 
	because $h$ is proper.
	\ Further $\ell$ is controlled since for any $U$ in $\cC_T$ we have $V:=(h\times h)(U)\in\cC_X$ by the controlledness of $h$, so by construction $(\iota\times\iota)^{-1}(V)$ 
	is one of the generators of $\iota^*\cC_X$.
	But $(\ell\times\ell)(U)\subseteq(\iota\times\iota)^{-1}(V)$, hence the claim.
\end{proof}

Again this generalizes immediately to \GgBC:
\begin{cor}\thlabel{tmp2}
	The category \GgBC~has all equalizers of pairs of morphisms.
\end{cor}
\begin{proof}
	Follows by construction of the equalizer together with \thref{gsubstr}.
\end{proof}

\begin{prop}\thlabel{coprod}
	The category $\gBC$ has all coproducts.
\end{prop}
\begin{proof}
	Let $(X_i)_{i\in I}$ be a family of generalized bornological coarse spaces. We define an object $X$ in $\gBC$ as 
	\[X:=\coprod_{i\in I}X_i, \qquad\!\!\! \cC_X:=\cC\left\langle\bigcup\nolimits_{i\in I}\cC_{X_i}\right\rangle,\qquad\!\!\! 
	\cB_X:=\left\langle\menge{B\subseteq X}{\forall i\in I:\ B\cap X_i\in\cB_{X_i}}\right\rangle.\]
	To prove that the structures $\cB_X$ and $\cC_X$ are compatible, it is enough to show that for all generators $U$ in $\cC_X$ the $U$-thickening of a generator of $\cB_X$ is bounded, so 
	assume $U\in\cC_{X_j}$ for some $j$ in $I$ and take a generator $B$ of $\cB_X$. Then $B_j:=B\cap X_j$ is bounded in $X_j$, hence $U[B]\cap X_j=U[B_j]$ is also bounded. 
	Further for any $k$ in $I$ with $k\neq j$
	we clearly have $U[B]\cap X_k=\leer$, thus $U[B]$ is a generator of $\cB_X$ and in particular it is bounded.\\
	We claim that $X$ represents the coproduct $\coprod_{i\in I}X_i$ in \gBC: The obvious inclusions $X_i\hookrightarrow X$ are clearly morphisms, and for an object $T$ in
	$\gBC$ together with morphisms $f_i\colon X_i\to T$ we get a set-theoretical unique map $f\colon X\to T$, so it suffices to show that this $f$ is a morphism in \gBC. 
	Since for all $i$ in $I$ the maps $f_i$ are controlled, the map $f$ maps generators of $\cC_X$ to entourages, and hence is controlled.
Now consider $B$ in $\cB_T$. For any $i$ in $I$ we have $f^{-1}(B)\cap X_i=f_i^{-1}(B)$, which is bounded by properness of $f_i$, hence
	by definition $f^{-1}(B)$ is a generator of $\cB_X$ and in particular it is bounded.
\end{proof}

\begin{prop}\thlabel{coequ}
	The category \gBC~has all co-equalizers of pairs of morphisms.
\end{prop}
\begin{proof}
	Let $f,g\colon X\to Y$ be two morphisms in \gBC. Define $E:=\operatorname{CoEq}(f,g)$ to be the set-theoretical co-equalizer of the underlying maps $f$ and $g$ and denote by
	$\pi$ the canonical map $Y\to E$. Define a coarse and a generalized bornological structure on $E$ as 
	\[\cC_E:=\cC\left\langle (\pi\times\pi)(\cC_Y)\right\rangle,\qquad \cB_E:=\cB\left\langle\menge{B\subseteq E}{\forall U\in\cC_E:\ \pi^{-1}\big(U[B]\big)\in\cB_Y}\right\rangle.\]
	By construction $\cC_E$ and $\cB_E$ are compatible and further the canonical map $\pi\colon Y\to E$ is controlled and proper. The latter can be seen as follows: 
	For a generator $B$ of $\cB_E$ 
	take $U:=\Delta_E$, then the pre-image of
	$U[B]=B$ is bounded by definition. The object $E$ together with the morphism $\pi$ represents the co-equalizer of $f$ and $g$ in \gBC:\\
	Consider an object $T$  in
	$\gBC$ together with a morphism $p\colon X'\to T$ such that $p\circ f=p\circ g$. Then set-theoretically this map factors through $E$: 
	\[\xymatrix{X\ar@<1mm>[r]^f\ar@<-1mm>[r]_g & Y\ar[r]^\pi\ar[dr]_p & E\ar@{.>}[d]^h\\ && T}\]
	By assumption the map $p$ is controlled, hence by commutativity of the diagram we immediately get that $h$ maps generators of $\cC_E$ to controlled sets. Therefore the map $h$ is
	controlled. To check that it is also proper, consider a bounded subset
	$B$ of $T$. We show, that $h^{-1}(B)$ is a generator of $\cB_E$. For this let $U$ be in $\cC_E$. One immediatelly verifies, that
	\[\pi^{-1}\big(U[h^{-1}(B)]\big)\subseteq p^{-1}\big((h\times h)(U)[B]\big).\]
	Therefore it remains to show, that the supset is bounded in $Y$:
	Since $h$ is controlled, $V:=(h\times h)(U)$ is an entourage of $T$, hence the $V$-thickening of $B$ is a bounded subset of $T$, so its pre-image under $p$ is bounded in $Y$
	which concludes the proof.
\end{proof}

\begin{cor}\thlabel{tmp3}
	The category \GgBC~has all coproducts and all co-equalizers of pairs of morphisms.
\end{cor}
\begin{proof}
	The proofs of \thref{coprod,coequ} generalize like in the dual cases for products and equalizers.
\end{proof}

\begin{thm}
	The categories \gBC~and \GgBC~are both complete and co-complete.
\end{thm}
\begin{proof}
	By \cite[Theorem V.2.1]{MacLane} a category having all products and equalizers of pairs of morphisms, is complete. The dual statement shows the co-completeness of a category having all
	coproducts and co-equalizers. Hence we are done by the above propositions and corollaries (\thref{prod,tmp1,coprod,coequ,tmp2,gprod,tmp3}).
\end{proof}

To make this more explicit we give concrete formulas for the limit and the colimit for arbitrary diagrams in \GgBC:

\begin{rmk}\thlabel{gen.colim}
	Let $\bI$ be a small category and \[D\colon\bI\to\GgBC,\qquad D'\colon\bI\to\GgBC\] be diagrams.
	We give explicit formulas for the (co-)limit of $D$, $D'$ respectively. Let $i\colon\GgBC\to\GSet$ be the forgetful functor and denote by $X:=\lim_{\bI} (i\circ D)$ the limit of
	$i\circ D$ in \GSet. We get for all $j$ in $\bI$ the canonical morphism $f_j\colon X\to i(D(j))$. We can consider the limit $\lim_{\bI} D$ as a subset of the product. 
	It is given by
	the set $X$ equipped with \vspace{0.1cm}
	\begin{align*}
		\cC_X&:=\cC\left\langle\menge{(X\times X)\cap\prod\nolimits_{j\in\bI}U_j}{\forall j\in\bI:\ U_j\in\cC_{D(j)}}\right\rangle\\[0.25cm]
		\cB_X&:=\cB\left\langle\menge{f_j^{-1}(B)}{j\in\bI,\ B\in\cB_{D(j)}}\right\rangle.
	\end{align*}\mbox{}\\[-0.25cm]
	Further denote $Y:=\colim_{\bI}(i\circ D')$ the colimit in \GSet. For all $j$ in $\bI$ we have the canonical 
	morphism $g_j\colon i(D'(j))\to Y$. Now the colimit $\colim_{\bI}D'$ in \GgBC~is given by
	$Y$ equipped with \vspace{1mm}
	\begin{align*}
		\cC_Y&:=\cC\left\langle\menge{(g_j\times g_j)(U)}{j\in\bI,\ U\in\cC_{D'(j)}}\right\rangle\\[0.25cm]
		\cB_Y&:=\cB\left\langle\menge{B\subseteq Y}{\forall U\in\cC_Y,\ \forall j\in\bI:\ g_j^{-1}\big(U[B]\big)\in\cB_{D'(j)}}\right\rangle.
	\end{align*}
\end{rmk}

Like in the category \GBC, the tensor product is part of a symmetric monoidal structure on \GgBC~with tensor unit the locally bounded 
one-point-space (i.e. a single point with maximal bornology, 
cf. \thref{tensor-symm}), 
but now we have an additional symmetric monoidal  structure:

\begin{lem}
	The category  \GgBC~is  symmetric monoidal with respect to the (Cartesian) product. The unit is the space consisting of a single unbounded point (i.e. the final object).
\end{lem}

\clearpage
\section{Connections between \texorpdfstring{\GBC}{GBC}~and \texorpdfstring{\GgBC}{GgBC}}
	In this section we study the connection of the classical category \GBC~and our new constructed \GgBC. It is easily seen that we can view \GBC~as full subcategory of \GgBC, but we have no adjunction
	between these two categories. Nevertheless we have strong connections of limits and colimits in both categories. 
	Further, there is a notion of coarse connectedness and coarse components. It will turn out that these components consist of only either bounded or unbounded points. 
	Finally there will be a criterion for when a diagram in the
	classical \GBC~admits a colimit.
	
\begin{prop}\thlabel{no.adj}
	The inclusion functor $i\colon\GBC\to\GgBC$ is a fully faithful embedding, but it has neither a left nor a right adjoint.
\end{prop}
\begin{proof}
	The first statement is clear from the definitions. Next, if $i$ had a right adjoint $R$, then $R$ preserves the final object. That contradicts 
	\thref{no.final}. On the other hand, assume $i$ has a left
	adjoint $L$. Then $L$ preserves colimits and in particular applied to the square in \thref{ex.PO} we get a pushout square
	\[
		\xymatrix{L(\N_{\min,\max})\ar[r]\ar[d]\ar@{}[dr]|(0.7){\text{\pigpenfont{I}}} & L(\N_{\max,\max}) \ar[d] \\ L(\N_{\min,\min})\ar[r] & L(\N_{\max,\emptyset}) }
	\]
	in \GBC. 
	However, since $i$ is fully faithful, the co-unit of the adjunction would give an natural isomorphism $L(i(X))\overset\cong\longrightarrow X$ for any $X$ in $\GBC$.
	Hence applying $i$ to the pushout square above gives the following pushout square in $\GBC$:
	\[
		\xymatrix{\N_{\min,\max}\ar[r]\ar[d]\ar@{}[dr]|(0.65){\text{\pigpenfont{I}}} & \N_{\max,\max} \ar[d] \\ \N_{\min,\min}\ar[r] & L(\N_{\max,\emptyset})}
	\]
	That is a contradiction to \thref{exa.N}.
\end{proof}

Consider a morphism $f\colon X\to Y$ of generalized bornological coarse spaces. The properness of $f$ provides a very usefull lemma:

\begin{lem}\thlabel{map.loc.small} \mbox{}\begin{enumerate}
		\item For any \emph{bounded} point $y$ in $Y$, the fiber $f^{-1}(y)$ contains only bounded points of $X$. \ In particular, if $Y$ is locally bounded, then so is $X$.
		\item If $X$ contains only unbounded points, then so does the image $f(X)$.
	\end{enumerate}
\end{lem}
\begin{proof}
	Both assertions follow immediately from the properness of $f$.
\end{proof}

\begin{defi}
	We define a relation $\sim_c$ on $X$ by: \ $x\sim_cy$ iff there exists an entourage $U$ in $\cC_X$ such that $(x,y)\in U$.
\end{defi}

\begin{lem}
	The relation $\sim_c$ on $X$ is an equivalence relation. 
\end{lem}
\begin{proof}
	By definition we have $\Delta_X\in\cC_X$, hence $x\sim_cx$ for all $x$. For symmetry assume $x\sim_cy$ is witnessed by some $U$ in $\cC_X$. Then also $U^{-1}\in\cC_X$, which gives
	$y\sim_cx$. Finally if $U,V$ in $\cC_X$ witness $x\sim_cy$ and $y\sim_cz$ respectively, then we have $U\circ V\in\cC_X$, which shows $x\sim_cz$.
\end{proof}

\begin{defi}
	The equivalence relation $\sim_c$ is called the \emph{coarse equivalence}. The equivalence classes are called the \emph{coarse components} of $X$. We denote the set of all coarse components
	of $X$ by $\pi_0^\coarse(X)$ and we call $X$ \emph{coarsely connected}, if $\pi_0^\coarse$ is a singleton.
\end{defi}

\begin{lem}\thlabel{thick.coarse}
	A coarse component of $X$ is either locally bounded or consists of unbounded points only.
\end{lem}
\begin{proof}
	Assume $M\subseteq X$ is a coarse component containing an unbounded point $x$ and a bounded point $y$ (i.e. $\{y\}$ is a bounded subset of $X$). By assumption there exists 
	some entourage $U$ such that $(x,y)\in U$, hence $x\in U[\{y\}]$, which is bounded, a contradiction.
\end{proof}

\begin{cor}\thlabel{iso.union}
	We get a canonical isomorphism $X\cong X_b\amalg X_h$ in \GgBC, where the subsets $X_b$ and $X_h$ are equipped with the subspace 
	structure (cf. \thref{subspace.str}).
\end{cor}
\begin{proof}
	The inclusions $X_b\hookrightarrow X$ and $X_h\hookrightarrow X$ are morphisms by construction and they induce the morphism $X_b\amalg X_h\to X$, whose underlying map
	of sets is the identity. \ Further  $i\colon X\to X_b\amalg X_h$ is controlled by definition of the coarse structure on the coproduct. Finally, since every non-empty 
	bounded subset in $X_b\amalg X_h$ is completely contained in $X_b$, the map $i$ is proper (by definition of the subspace structure on $X_b$).
\end{proof}

In the following $\iota$ denotes the fully faithful inclusion functor
\[
	\iota\colon\GBC\longrightarrow\GgBC.
\]
Even though $\iota$ has neither left nor right adjoint (cf. \thref{no.adj}), it preserves all limits and colimits:

\begin{prop}\thlabel{pres.lim}
	Let $\bI$ be a small category and $D\colon \bI\to\GBC$ be a diagram, whose limit exists in \GBC. 
	Then the canonical map \[\iota(\lim\nolimits_{\bI} D)\longrightarrow\lim\nolimits_{\bI} \iota(D)\] is an isomorphism. Similarly if $D'\colon\bI\to\GBC$ is a diagram
	whose colimit exists in \GBC, then the canonical map \[\colim\nolimits_{\bI} \iota(D')\longrightarrow \iota(\colim\nolimits_{\bI} D')\] is an isomorphism.
\end{prop}
\begin{proof}
	Let $L:=\lim_{\bI} D$ in \GBC. Then $\iota(L)$ fulfills the universal property of the limit of the diagram $\iota\circ D$: In fact for any object 
	$T$ in $\GgBC$ together with morphisms $f_i\colon T\to D(i)$ each\footnote{note that the limit is not taken over the empty diagram, because this does not have a limit in \GBC~as shown
		in \thref{no.final}}
	of the morphisms $f_i$ imply that $T$ is locally bounded by \thref{map.loc.small}. Therefore we can view the limit of $\iota\circ D$ as an object in \GBC~(which embedds fully faithful into \GgBC), hence we are done.\\
	Now let $L':=\colim_{\bI} D'$ be the colimit in \GBC. On the other hand, since \GgBC~is 
	co-complete, there exists a colimit $C:=\colim_{\bI} \iota\circ D'$. By the universal property of $C$, there exists a unique morphism $C\to \iota(L)$.
	But $\iota(L)$ is locally bounded, hence by \thref{map.loc.small} also $C$ is locally bounded, and therefore can be viewed as object in \GBC. By the universal propery of $L$ in 
	\GBC~we get $L\cong C$ and we are done.
\end{proof}

\begin{lem}\thlabel{lim.recog}
	Let $\bI$ be a small category and $D\colon\bI\to\GBC$ be a diagram. The limit of this diagram exists iff the limit of $\iota\circ D$ in \GgBC~is locally bounded. \\The analogous statement
	holds for colimits.
\end{lem}
\begin{proof}
	If the limit of $D$ exists in \GBC, then it is isomorphic to the limit of  $\iota\circ D$
	in \GgBC~(by \thref{pres.lim}) and hence the latter is locally bounded. 
	On the other hand if there exists a locally bounded limit of $\iota\circ D$ in \GgBC, then
	we can view it as an object in \GBC~and this object fulfills the universal property of $\lim_{\bI}D$ as \GBC~is a full subcategory of \GgBC.
\end{proof}

\begin{cor}\thlabel{cl.almost.compl}
	The category \GBC~has all non-empty limits.
\end{cor}
\begin{proof}
	Let $D\colon\bI\to\GBC$ be a diagram for a small and non-empty category~$\bI$. Let $L:=\lim_{\bI}\iota\circ D$ be the limit in \GgBC~and choose some $j$ in $\bI$, then the projection
	$L\to \iota(D(j))$ witnesses that $L$ is locally bounded (by \thref{map.loc.small}). \thref{lim.recog} yields the claim.
\end{proof}

\thref{lim.recog} gives also a characterization for which diagrams in \GBC~have a colimit. This was already conjectured by the authors of \cite{Uli}. 
We let denote by $F$ the forgetful functor $F\colon\GBC\to\GSet$.
\begin{cor}
	Let $\bI$ be a small category and let $D\colon\bI\to\GBC$ be a diagram. Denote by $X_i$ the object $D(i)$ in $\GBC$.
	We write $X$ for the colimit of the diagram $F\circ D$ in $\GSet$. Furthermore, let $f_i$ denote the canonical maps $F(X_i)\to X$. \ The colimit of the diagram
	$D$ exists in \GBC, iff the diagram is \emph{colimit admissible}, i.e. 
	\begin{align*}
	\tag{*}\forall b\in X\ \forall n\in\N\ &\forall k,i_1,\dots,i_n\in\bI,\ 
	\forall (U_j)_{j\in\{i_1,\dots,i_n\}}\in\textstyle\prod_j\cC_{X_{i_j}}: \\ &f_k^{-1}\Big((f_{i_n}\times f_{i_n})(U_n)\big[\cdots (f_{i_1}\times f_{i_1})(U_1)[\{b\}]\big]\Big)
	\in\cB_{X_k}.
	\end{align*}
\end{cor} 
\begin{proof}
	First assume that the colimit exists in $\GBC$. Since its underlying set needs to be $X$ (\thref{adjoint}), we refer to the colimit in $\GBC$ also by $X$.
	Also the maps $f_i$ become morphisms $X_i\to X$. \ We want to verify (*):\\
	For any $b$ in $X$, the subset $\{b\}$ is bounded by definitions. Furthermore -- since the maps $f_{i_j}$ are controlled and the subsets $U_{i_j}$ are entourages --
	the subsets $(f_{i_j}\times f_{i_j})(U_{i_j})$ are entourages on $X$. Therefore the successively thickened subset in (*) is still bounded in $X$. So its pre-image under $f_k$
	is bounded in $X_k$, which shows one direction of the claim.\\[.5em]
	On the other
	hand assume $D$ is colimit admissible, i.e. assume that (*) holds. 
	First consider the diagram $F\circ D$. Since $\GgBC$ is co-complete, it has a colimit, which we denote by $L$. \ 
	By \thref{lim.recog} we are done, if we verify, that $L$ is locally bounded. For this we
	choose an arbitrary point $b$ in $L$\footnote{Note that the underlying set of $L$ needs to be $X$ since the forgetful functor preserves colimits.}. We want to show, that
	$\{b\}\in\cB_L$. Recalling the definition of $\cB_L$ in \thref{gen.colim}, we have to verify, that for all $k$ in $\bI$ and for all $U$ in $\cC_L$, the pre-image $f_k^{-1}(U[b])$
	is bounded in $X_k$.
	The coarse structure $\cC_L$ is generated by subsets $(f_i\times f_i)(U)$ for $i$ in $\bI$ and $U$ in $\cC_{X_i}$. Therefore the condition (*) precisely ensures that
	for an \emph{arbitrary} entourage $U$ in $\cC_L$ we have $f_k^{-1}(U[b])\in\cB_{X_k}$.
\end{proof}

\clearpage
\section{Coarse homology theory}
	The authors in \cite{Uli} elaborate the notion (and generalize it to the equivariant situation in \cite{equ.Uli}) of 
	motivic coarse spectra on which large parts of their following work bases on. We imitate these definitions to construct the notion of 
	motivic
	coarse spectra over generalized bornological coarse spaces. The aim of this section is to show, that motivic coarse spectra over \GBC~are equivalent to those over \GgBC.
	
	We just give the necessary sequence of definitions needed for defining motivic coarse spaces and spectra. We will not go into details or give instructive examples. Both can be richly found
	in \cite{Uli}. In particular all stated assertions without a proof can be found proven in that paper.

	\subsection{Basic notions}
		In this subsection we briefly give some notions for generalized bor\-no\-lo\-gical coarse spaces needed for the following.

	Let $X$ be an object in $\GgBC$.
	
	\begin{defi}
		A \emph{$\Gamma$-equivariant big family} 
		$\cY$ on $X$ is a filtered family of $\Gamma$-invariant subsets $(Y_i)_{i\in I}$ of $X$ such that for every $i$ in $I$ and every entourage $U$ in $\cC_X$ there exists an $j$ in $I$ 
		such that
		$U[Y_i]\subseteq Y_j$.\\
		A \emph{$\Gamma$-equivariant complementary pair} on $X$ is a pair $(Z,\cY)$ consisting of a $\Gamma$-invariant subset $Z$ of $X$ and a 
		$\Gamma$-equivariant big family $\cY$ on $X$ such that there exists an $i$ in $I$ with $Z\cup Y_i=X$.
	\end{defi}
	
	Consider two morphisms $f,g\colon X\to Y$ between $\Gamma$-equivariant generalized bornological coarse spaces. Then $f$ and $g$ are said to be \emph{close to each other}, if
	$(f\times g)(\Delta_X)\in\cC_Y$.
	
	\begin{rmk}
		The morphisms $f$ and $g$ are close to each other iff $h\colon\{0,1\}_{\max,\max}\otimes X\to Y$ defined by $h(0,x)=f(x)$ and $h(1,x)=g(x)$ for all $x$ in $X$, is a morphism.
	\end{rmk}
	
	\begin{defi}
		The morphism $f\colon X\to Y$ is an \emph{equivalence} of $\Gamma$-equivariant generalized 
		bornological coarse spaces, if there exists a morphism $f'\colon Y\to X$ such that $f\circ f'$ and $f'\circ f$ are close
		to the respective identities. \ In this case we write $X\simeq_c Y$.
	\end{defi}
	
	\begin{defi}
		The generalized bornological coarse space $X$ is called \emph{flasque} if it admits a morphism $f\colon X\to X$ with the following properties:
		\begin{enumerate}
			\item The morphisms $f$ and $\mathrm{id}_X$ are close to each other.
			\item For any entourage $U$ in $\cC_X$ the union $\bigcup_{k\in\N}(f^k\times f^k)(U)$ is  an entourage of $X$.
			\item For any bounded subset $B$ in $\cB_X$ there exists some $k$ in $\N$ such that $f^k(X)\cap\Gamma B=\emptyset$.
		\end{enumerate}
	\end{defi}
	
	\begin{lem}\thlabel{huge.is.flasque}
		For any space $X$ in $\GgBC$ the subspace $X_h$ of unbounded points is flasque.
	\end{lem}
	\begin{proof}
		The identity $\mathrm{id}_{X_u}$ witnesses the flasqueness: The conditions (1)+(2) are obviously fulfilled and since the only bounded subset is the empty set, we conclude the proof.
	\end{proof}
	
Let $U$ in $\cC_X^\Gamma$ be a $\Gamma$-invariant entourage. 
If we replace $\cC_X$ by the coarse structure generated by $U$, then this 
coarse structure is clearly compatible with $\cB_X$ and defines a new $\Gamma$-equivariant generalized bornological
coarse space, which we denote by $X_U$.
	
	\begin{rmk}
		The natural morphisms $X_U\to X$ induce an isomorphism \[ \colim_{U\in\cC_X^\Gamma}X_U\iso X\]
		in \GgBC.
	\end{rmk}

	Consider a co-complete stable $\infty$-category $\cC$ and a functor $F\colon\GgBC\to\cC$. 
		For a $\Gamma$-equivariant big family $\cY$ on $X$ we define $F(\cY):=\colim_{i\in I}F(U_i)$.

	\begin{defi}\thlabel{def.class.h.t}
		The functor $F\colon\GgBC\to\cC$ is called a \emph{$\cC$-valued coarse homology theory}, if:
		\begin{enumerate}
			\item (\emph{Excision}) For all $\Gamma$-equivariant complementary pairs $(Z,\cY)$ on a space $X$ in $\GgBC$, the following square is co-Cartesian:
			\[
				\xymatrix{
					F(Z\cap\cY)\ar[r]\ar[d] & F(Z)\ar[d]\\
					F(\cY)\ar[r] & F(X).
			}
			\]
			\item (\emph{Coarse invariance}) If $X\to X'$ is a coarse equivalence in $\GgBC$, then the induced morphism $F(X)\to F(X')$ is an equivalence in $\cC$.
			\item (\emph{vanishing on flasques}) If $X$ is a flasque generalized bornological coarse space, then the canonical map $0\to F(X)$ is an equivalence in $\cC$.
			\item (\emph{$U$-continuity}) For every $X$ in $\GgBC$ the collection of morphisms $X_U\to X$ induces an equivalence 
			\[
				\colim_{U\in\cC_X^\Gamma}F(X_U)\to F(X).
			\]
		\end{enumerate}
	\end{defi}

\subsection{Generalized motivic coarse spaces}
	In this subsection we elaborate on motivic coarse spaces. Essentially we just copy all definitions from \cite{Uli} and \cite{equ.Uli}. 
	If one reads this entire subsection without tildes, one gets the definitions
	of motivic coarse spaces on \GBC.
	
	We write \textbf{sSet} for the category of simplicial sets and we denote by $W$ the class of weak homotopy equivalences. We obtain the 
	\emph{$\infty$-category of spaces} $\mathbf{Spc}:=\mathbf{sSet}[W^{-1}]$. \\Further, for any $\infty$-category $\mathbf{C}$ we denote by $\mathbf{PSh}(\mathbf{C}):=
	\mathbf{Fun}(\mathbf{C}^{op},\mathbf{Spc})$ the $\infty$-category of \emph{space-valued presheaves} on \textbf{C} and by \[\mathrm{yo}\colon\mathbf{C}\to\mathbf{PSh}(\mathbf{C})\] the 
	Yoneda-embedding.

\begin{lem}\thlabel{Gro.Top}
	There is a subcanonical Grothendieck topology $\tau_\chi$ on \GgBC~such that the $\tau_\chi$-sheaves are exactly 
	those presheaves $E$ on \GgBC~which satisfy $E(\leer)\simeq\nolinebreak\ast$ and such that for  any $\Gamma$-equivariant complementary pair $(Z,\mathcal Y)$
	the following square is cartesian in \textbf{\upshape Spc} (in which case we say, that $E$ \emph{fulfills descent w.r.t. $(Z,\cY)$}).
	\[
		\xymatrix{E(X)\ar[r]\ar[d] & E(Z)\ar[d]\\ E(\mathcal Y)\ar[r] & E(Z\cap\mathcal Y).}
	\]
\end{lem}
\begin{proof}
	Let $(Z,\cY)$ be a $\Gamma$-equivariant complementary pair.\\
	\textit{Claim 1}:\quad $E$ fulfills descent for $(Z,\cY)$ iff $E(X)\to\lim_{W\in\sS_{(Z,\cY)}}E(W)$ is an equivalence.\\
	\textit{Proof}: \ 
	We define the following four collections of morphisms:
	\begin{alignat*}{3}
		\sS_{(Z,\cY)}&:=\left\{f\colon W\to X\ \middle|\ \tiny\begin{matrix}f\text{ factors through}\\Z\text{ or some }Y_i\end{matrix}\right\} &&\qquad\qquad 
		\sS_Z&&:=\big\{W\to Z\big\}\\
		\sS_\cY&:=\big\{W\to Y_i\ \big|\ Z_i\in\cY\big\} &&\qquad\qquad \sS_{\cY\cap Z}&&:=\big\{W\to Y_i\cap Z\ \big|\ Y_i\in\cY\big\}
	\end{alignat*}
	We obtain a pushout-diagram 
	\[
		\xymatrix{\sS_{\cY\cap Z}\ar[r]\ar[d]\ar@{}[dr]|(0.65){\text{\pigpenfont{I}}} & \sS_Z\ar[d]\\
			\sS_\cY\ar[r] & \sS_{(Z,\cY)}}
	\]
	and thus we have 
	\[
		\Fun(\sS_{(Z,\cY)},\Spc)\simeq\Fun(\sS_Z,\Spc)\times_{\Fun(\sS_{\cY\cap Z},\Spc)}\Fun(\sS_\cY,\Spc). \tag{*}
	\]
	Now we claim, that for any presheaf $E$ on $\GgBC$ the following square is a pullback:
	\[\xymatrix{
		\lim\limits_{W\in\sS_{(Z,\cY)}}E(W)\ar[r]\ar[d] & \lim\limits_{W\in\sS_Z}E(W)\ar[d]\\
		\lim\limits_{W\in\sS_\cY}E(W)\ar[r] & \lim\limits_{W\in\sS_{\cY\cap Z}}E(W).}
	\]
	To show this, let $A$ in $\Spc$ and denote by $\underline A$ the constant-$A$-functor. Then we calculate
	\begin{align*}
		\Map_\Spc\left(A,\lim\nolimits_{\sS_{(Z,\cY)}}E\right)&\simeq\Map_{\Fun(\sS_{(Z,\cY)},\Spc)}\left(\underline A, E\right)\\
		&\underset{(*)}{\simeq}\Map_{\Fun(\sS_\cY,\Spc)}\left(\underline A,E\right)\times_{\Map_{\Fun(\sS_{\cY\cap Z},\Spc)}\left(\underline A,E\right)}
		\Map_{\Fun(\sS_Z,\Spc)}\left(\underline A,E\right)\\
		&\simeq \Map\left(A,\lim\nolimits_{\sS_\cY}E\right)\times_{\Map\left(A,\lim_{\sS_{\cY\cap Z}}E\right)}\Map\left(A,\lim\nolimits_{\sS_Z}E\right).
	\end{align*}
	Next, we observe that $Z\to Z$ is final in $\sS_Z$, hence $\lim_{W\in\sS_Z}E(W)\simeq E(Z)$ and likewise we get
	$\lim_{W\in\sS_\cY}E(W)\simeq\lim_{i\in I}E(Y_i)$ and $\lim_{W\in\sS_{\cY\cap Z}}E(W)\simeq\lim_iE(Y_i\cap Z)$. We therefore obtain the following
	pullback-square:
	\[
		\xymatrix{
			\lim\limits_{W\in\sS_{(Z,\cY)}}E(W)\ar[r]\ar[d]\ar@{}[dr]|(0.35){\text{\pigpenfont{A}}} & E(Z)\ar[d]\\
			\lim\limits_iE(Y_i)\ar[r]&\lim\limits_iE(Y_i\cap Z).
		}
	\]
	By universal property of the pullback, the inclusions $Z\to X$ and $Y_i\to X$ induce a morphism $E(X)\to\lim_{\sS_{(Z,\cY)}}E(W)$, which clearly is an equivalence if and only
	if the presheaf $E$ fulfills decent for $(Z,\cY)$.\hfill $\Box_{\text{Claim 1}}$\\[.3cm]
	For a presheaf satisfying descent for $(Z,\cY)$ we let $\tau_E$ be the finest Grothendieck topology on $\GgBC$ such that $E$ is a $\tau_E$-sheaf. Further we define
	\[\tau_\Box:=\bigcap_{\substack{E\text{ fulfills descent}\\\text{w.r.t. all }(Z,\cY)}}\tau_E.\]
	It is the finest Grothendieck topology such that all presheaves satisfying descent w.r.t. all complementary pairs, are $\tau_\Box$-sheaves.\\
	Next, we let $\tau_\cS$ denote the Grothendieck topology generated by all sieves $\sS_{(Z,\cY)}$ for all complementary pairs $(Z,\cY)$.\\
	By Claim 1, we have $\tau_\cS\subseteq\tau_\Box$ and therefore $\mathrm{Sh}^{\tau_\Box}\subseteq\mathrm{Sh}^{\tau_\cS}$.\\
	Again, by Claim 1, we have \[\mathrm{Sh}^{\tau_\cS}\subseteq\big\{E\in\PSh(\GgBC)\ \big|\ E\text{ fulfills descent w.r.t all }(Z,\cY)\big\}=:\mathcal M.\]
	Finally, by definition we have $\mathcal M\subseteq\mathrm{Sh}^{\tau_\Box}$. All together we obtain 
	\[
		\mathrm{Sh}^{\tau_\Box}\subseteq\mathrm{Sh}^{\tau_\cS}\subseteq\mathcal M\subseteq\mathrm{Sh}^{\tau_\Box}.
	\]
	Thus, we can choose $\tau_\chi$ as $\tau_\Box$ and the upper chain of inclusions shows that $\tau_\Box$-sheaves are precisely
	those presheaves lying in $\mathcal M$, i.e. fulfilling descent for any complementary pair.\\[.3cm]
	To see that $\tau_\chi$ is subcanonical, we refer to \cite[Lemma 3.12]{Uli}.
\end{proof}

	Let $\mathbf{Sh}(\GgBC)$ denote the full subcategory of $\tau_\chi$-sheaves in $\mathbf{PSh}(\GgBC)$. Then we obtain the usual sheafification adjunction \[
		\adjunction{\tilde L}{\mathbf{PSh}(\GgBC)}{\mathbf{Sh}(\GgBC)}{\mathrm{incl}}.
	\]

Consider a sheaf $E$ in $\Sh(\GgBC)$ and let $\mathbb{I}:=\{0,1\}_{\max,\max}$.

\begin{defi}
	The sheaf $E$ is called \emph{coarsely invariant} if for all generalized bornological coarse spaces $X$ in $\GgBC$, the projection $\mathbb{I}\otimes X\to X$ induces
	an equivalence \[E(X)\overset\simeq\longrightarrow E\left(\mathbb{I}\otimes X\right)\] in $\mathbf{Spc}$.
\end{defi}

\begin{lem}
	The collection of all coarsely invariant sheaves form a full localizing subcategory $\Sh^\I(\GgBC)$ of $\Sh(\GgBC)$. We obtain an adjunction
	\[
		\adjunction{\tilde{H}^\I}{\Sh(\GgBC)}{\Sh^\I(\GgBC)}{\mathrm{incl}}.
	\]
\end{lem}
\begin{proof}
	The subcategory of coarsely invariant sheaves is the full subcategory of objects which are local with respect to the collection of morphisms
	$\mathrm{yo}(\I\otimes X)\to\mathrm{yo}(X)$. Note that these morphisms are morphisms of sheaves by \thref{Gro.Top}.
	Now the claim follows from \cite[Prop 5.5.4.15]{HTT}.
\end{proof}

\begin{rmk}
	To show that a sheaf is coarsely invariant, it suffices to show that morphisms, which are close to each other are mapped to equivalent morphisms.
\end{rmk}

\begin{defi}
	A sheaf $E$ in $\Sh^{\I}(\GgBC)$ \emph{vanishes on flasques} if $E(X)\simeq\ast$ for any flasque generalized $\Gamma$-bornological coarse space $X$. 
\end{defi}
\begin{lem}
	The collection of all coarsely invariant sheaves that vanish on flasques form
	a full localizing subcategory $\Sh^{\I,\mathrm{fl}}(\GgBC)$ of $\Sh^{\I}(\GgBC)$. We get a corresponding adjunction 
	\[
		\adjunction{\tilde{\mathrm{Fl}}}{\Sh^{\I}(\GgBC)}{\Sh^{\I,\mathrm{fl}}(\GgBC)}{\mathrm{incl}}.
	\]
\end{lem}
\begin{proof}
	A sheaf  in $\Sh(\GgBC)$ vanishes on a flasque generalized $\Gamma$-bornological coarse space $X$ iff it is local w.r.t. $\mathrm{yo}(\leer)\to\mathrm{yo}(X)$ (see 
	\cite[Rem. 3.24]{Uli}). Therefore:
	The category of coarsely invariant sheaves which vanish on flasques is the full subcategory of those sheaves in $\Sh^{\I}(\GgBC)$ which are local with respect to
	the collection of morphisms $\mathrm{yo}(\leer)\to\mathrm{yo}(X)$.
\end{proof}

\begin{defi}
	We say that a sheaf $E$ in $\Sh^{\I,\mathrm{fl}}(\GgBC)$ is \emph{$u$-continuous} if for every generalized bornological coarse space $X$ in $\GgBC$ the collection of natural morphisms
	$\{X_U\to X\}_{U\in\cC_X^\Gamma}$ induces an equivalence
	\[
		E(X)\overset\simeq\longrightarrow\lim_{U\in\cC_X^\Gamma}E(X_U)
	\]
	in $\mathbf{Spc}$. The full subcategory of all $u$-continuous sheaves in $\Sh^{\I,\mathrm{fl}}(\GgBC)$ is called the category of \emph{generalized motivic coarse spaces} and is 
	denoted by $\GgSpc$.
\end{defi}
\begin{lem}
	The full subcategory $\GgSpc$ of $\Sh^{\I,\mathrm{fl}}(\GgBC)$ is localizing and fits into an adjunction
	\[
		\adjunction{\tilde U}{\Sh^{\I,\mathrm{fl}}(\GgBC)}{\GgSpc}{\mathrm{incl}}.
	\]
\end{lem}
\begin{proof}
	Similarly as in the proofs above, we argue that $\GgSpc$ is the full subcategory of objects in $\Sh^{\I,\mathrm{fl}}(\GgBC)$ which are local w.r.t. the collection of morphisms
	$\colim_{\scriptscriptstyle U\in\cC_X^\Gamma}\mathrm{yo}(X_U)\to\mathrm{yo}(X)$ for all $X$ in $\GgBC$.
\end{proof}

\begin{defi}
	We define the composition of the functors above as
	\[
		\mathrm{Y\tilde o}:=\tilde U\circ \tilde{\mathrm{Fl}}\circ \tilde{H}^{\I}\circ \tilde L\circ \mathrm{yo}\colon\GgBC\longrightarrow\GgSpc.
	\]
\end{defi}

We collect and summarize this subsection in the following corollary.

\begin{cor}\thlabel{Yo.prop}
	\begin{enumerate}
		\item The $\infty$-category of motivic coarse spaces is presentable and fits into a localization
		\[
			\adjunction{\tilde U\circ \tilde{\mathrm{Fl}}\circ \tilde{H}^{\I}\circ\tilde L}{\mathbf{PSh}(\GgBC)}{\GgSpc}{\mathrm{incl.}}
		\]
		\item If $(Z,\mathcal Y)$ is a complementary pair on a space $X$ in $\GgBC$, then the square
			\[
				\xymatrix{\gYo(Z\cap\cY)\ar[d]\ar[r] & \gYo(\cY)\ar[d]\\
					\gYo(Z)\ar[r] & \gYo(X)}
			\] is co-Cartesian in \GgSpc.
			\item If $X\to X'$ is an equivalence of $\Gamma$-equivariant generalized bornological coarse spaces, then the morphism $\gYo(X)\to\gYo(X')$ is an equivalence in \GgSpc.
			\item If $X$ is flasque in \GgBC, then $\gYo(X)$ is final in \GgSpc.
			\item For every $X$ in $\GgBC$ the canonical morphism defines an equivalence 
			\[\colim_{U\in\cC_X^\Gamma}\gYo(X_U)\overset\simeq\longrightarrow\gYo(X).\]
	\end{enumerate}
\end{cor}

\begin{cor}\thlabel{UE.1}
	For every co-complete $\infty$-category $\mathbf{C}$ \ precomposition with $\mathrm{Y\tilde o}$ induces an equivalence between the category of colimit-preserving functors $\mathbf{Fun}^{colim}(\GgSpc,\mathbf{C})$ and the full subcategory of
	$\mathbf{Fun}(\GgBC,\mathbf{C})$ of functors which satisfy excision, preserve coarse equivalences, annihilate flasque spaces and are $u$-continuous.
\end{cor}

Finally there is an excision statement for coarsely excisive pairs. For that we first need a notion of \enquote{nice} subsets of $X$: 
\begin{defi}
	A $\Gamma$-invariant subset 
	$A$ of $X$ is called \emph{nice} if for every $\Gamma$-invariant entourage $U$ in $\cC_X$ with $\Delta_X\subseteq U$, the inclusion $A\to U[A]$ is a coarse equivalence.
\end{defi}

Let $Y$ and $Z$ be $\Gamma$-invariant subsets of $X$.
\begin{defi}
	We say, that $(Y,Z)$ are \emph{coarsely excisive}, if \begin{enumerate}
		\item $X=Y\cup Z$.
		\item For all entourages $U$ in $\cC_X$ there exists an entourage $V$ in $\cC_X$ such that
		\[
		U[Y]\cap U[Z]\subseteq V[Y\cap Z].
		\]
		\item There exists a cofinal subset $\cC_X^{\mathrm{nice}}$ of $\cC_X^\Gamma$ such that for all $U$ in $\cC_X^{\mathrm{nice}}$ the set $U[Y]\cap Z$ is nice in $X$.
	\end{enumerate} 
\end{defi}

\begin{exam}\thlabel{coprod.excisive}
	Consider two generalized $\Gamma$-bornological coarse spaces $X$ and $Y$. We identify them with their respective image in the coproduct $Z:=X\amalg Y$. Then the pair $(X,Y)$
	is coarsely excisive for $Z$.
\end{exam}
\begin{proof}
	We claim that for an arbitrary entourage $U$ in $\cC_Z$ we have $U[X]\subseteq X$. The same argument of course shows $U[Y]\subseteq Y$. We therefore obtain
	$U[X]\cap U[Y]=\leer$ and thus conditions (2) und (3) are fulfilled by trivial reasons. Hence we only have to verify our claim above:\\
	Consider an entourage $U$ in $\cC_Z$. If $U\in\cC_X$, then clearly $U[X]\subseteq X$. On the other hand, if we have $U\in\cC_Y$, then we have $U[X]=\leer$, because $X\cap Y=\leer$.
	We know that $\cC_Z$ is generated by the entourages on $X$ and those on $Y$ and we have just verified the claim for the generators. It remains to check unions and compositions of
	those. For these the claim follows from the formulas $(U\cup V)[X]=U[X]\cup V[X]$ and $(U\circ V)[X]=U[V[X]]$.
\end{proof}

\begin{cor}\thlabel{huge.small.excise}
	The pair $(X_b,X_h)$ is coarsely excisive.
\end{cor}
\begin{proof}
	This is an immediate consequence of \thref{iso.union} and \thref{coprod.excisive}
\end{proof}

\begin{lem}\thlabel{excision}
	If $(Y,Z)$ is a coarsely excisive pair on $X$, then the square
	\[
		\xymatrix{
			\gYo(Y\cap Z)\ar[r]\ar[d] & \gYo(Y)\ar[d]\\\gYo(Z)\ar[r] & \gYo(X)
		}
	\]
	is co-Cartesian in \GgSpc.
\end{lem}
\begin{proof}
	See \cite[Corollary 4.13]{equ.Uli}.
\end{proof}

\begin{cor}\thlabel{Xadd}
	Consider two generalized $\Gamma$-bornological coarse spaces $Y$ and $Z$. Then by \thref{coprod.excisive} the pair $(Y,Z)$ is coarsely excisive the space $Y\amalg Z$. Hence
	\thref{excision} gives an equivalence
	\[
		\gYo(Y)\oplus\gYo(Z)\overset\simeq\longrightarrow\gYo(Y\amalg Z).
	\]
\end{cor}

\subsection{Generalized motivic coarse spectra}
	As a final step in the construction we now stabilize generalized coarse motivic spaces and obtain generalized coarse motivic spectra. Like in
	the previous subsection, we just copy the definitions from \cite{Uli}, and
	reading the subsection without tildes gives the definitions of (ordinary) motivic coarse spectra.
	
	\begin{defi}
		We denote by $\GgSpc_*:=\GgSpc/\ast$ the pointed version of generalized motivic coarse spaces. Further we define the \emph{suspension functor}
		\[\Sigma\colon\GgSpc_*\to\GgSpc_*,\qquad X\mapsto\colim(\ast\leftarrow X\rightarrow\ast).\]
		The category of \emph{generalized motivic coarse spectra} is defined by inverting $\Sigma$:
		\[
			\GgSp:=\GgSpc[\Sigma^{-1}]:=\colim\{\GgSpc_*\overset\Sigma\rightarrow\GgSpc_*\overset\Sigma\rightarrow\dots\}
		\]
		where the colimit is taken in the category of presentable $\infty$-categories together with left adjoint functors.
	\end{defi}

By construction $\GgSp$ is a presentable stable $\infty$-category. 
Precomposing  the canonical functor $\GgSpc_*\to\GgSp$ with the functor, which adds a disjoint base point gives a functor
\[\Sigma_+^{\mathrm{mot}}\colon \GgSpc\longrightarrow\GgSpc_*\longrightarrow\GgSp.\]

\begin{lem}\thlabel{Sigma.adj}
	The functor $\Sigma_+^{\mathrm{mot}}$ fits into an adjunction
	\[
		\adjunction{\Sigma_+^{\mathrm{mot}}}{\GgSpc}{\GgSp}{\Omega^{\mathrm{mot}}}
	\]
	and has the following universal property: \ For every co-complete stable $\infty$-category $\mathbf{C}$, precomposing with $\Sigma_+^{\mathrm{mot}}$ provides an equivalence
	\[
		\mathbf{Fun}^{\colim}(\GgSpc,\mathbf{C})\simeq\mathbf{Fun}^{\colim}(\GgSp,\mathbf{C}).
	\]
\end{lem}
\begin{proof}
	See \cite[31]{Uli}.
\end{proof}

Every $\Gamma$-equivariant generalized bornological coarse space $X$ represents a generalized coarse motivic spectrum via the composition of functors
\[\gYo^s\colon\GgBC\overset{\gYo}{\longrightarrow}\GgSpc\overset{\Sigma^{\mathrm{mot}}_+}\longrightarrow\GgSp.\]

For a stable $\infty$-category $\mathbf C$ any morphism $f\colon E\to F$ in $\mathbf C$ extends functorially to a cofiber sequence
\[
	\dots\to\Sigma^{-1}E\to\Sigma^{-1}F\to\Sigma^{-1}\operatorname{Cofib}(f)\to E\to F\to\operatorname{Cofib}(f)\to\Sigma E\to\Sigma F\to\dots.\]
We denote $\operatorname{Fib}(f):=\Sigma^{-1}\operatorname{Cofib}(f)$.

For a big family $\cY=(Y_i)_{i\in I}$ on a $\Gamma$-equivariant generalized bornological coarse space $X$ we define the generalized motivic coarse spectrum
\[\gYo^s(\cY):=\colim_{i\in I}\gYo^s(Y_i).\]
The induced natural morphism $\gYo^s(\cY)\to\gYo^s(X)$ gives rise to 
\[
	(X,\cY):=\operatorname{Cofib}\left(\gYo^s(\cY)\to\gYo^s(X)\right).
\]

Since $\Sigma_+^{\mathrm{mot}}$ is left adjoint (cf. \thref{Sigma.adj}), an immediate consequence of the properties of $\gYo$ like in \thref{Yo.prop} or \thref{excision} is the following:

\begin{cor}
	\begin{enumerate}
		\item We have a cofibre sequence
		\[
			\dots\to\gYo^s(\cY)\to\gYo^s(X)\to(X,\cY)\to\Sigma\gYo^s(\cY)\to\dots
		\]
		\item For a complementary pair $(Z,\cY)$ on $X$, the natural morphism \[(Z,Z\cap\cY)\longrightarrow(X,\cY)\] is an equivalence.
		\item If $X\to Y$ is an equivalence in $\GgBC$, then $\gYo^s(X)\to\gYo^s(Y)$ is an equivalence in $\GgSp$.
		\item If $X$ is flasque, then $\gYo^s(X)$ is a zero-object in $\GgSp$.
		\item We have an equivalence \ $\gYo^s(X)\simeq\colim_{U\in\cC_X^\Gamma}\gYo^s(X_U)$.
		\item If $(Y,Z)$ is a coarsely excisive pair on $X$, then the square
		\[\xymatrix{
			\gYo^s(Y\cap Z)\ar[r]\ar[d] & \gYo^s(Y)\ar[d]\\ \gYo^s(Z)\ar[r] & \gYo^s(X)
		}\]
			is co-cartesian in $\GgSp$.
	\end{enumerate}
\end{cor}

\begin{cor}
	Since $(X_b,X_h)$ is coarsely excisive by \thref{huge.small.excise} and $X_h$ is flasque by \thref{huge.is.flasque}, we have 
	$\gYo^s(X)\simeq\gYo^s(X_b)\oplus\gYo^s(X_h)\simeq\gYo^s(X_b)$.
\end{cor}

\subsection{Equivalence of \texorpdfstring{\GgSp}{GgSp}~ and \texorpdfstring{\GSp}{GSp}}
	In this subsection we show that the inclusion $\iota\colon\GBC\to\GgBC$ induces an equivalence of categories $\GgSp\overset\simeq\to\GSp$. 
	
	Let $\Sp$ denote the stable $\infty$-category of spectra (i.e. spectrum objects in $\Spc_\ast$). For any category $\mathbf{C}$ we denote by
	$\PShSp(\mathbf{C}):=\mathbf{Fun}(\mathbf{C}^{op},\Sp)$ the $\infty$-category of spectrum-valued presheaves on $\mathbf{C}$. This category is equivalent to
	the subcategory of spectrum objects in $\PSh(\mathbf{C})$ by \cite[Rmk 1.4.2.9]{HA}.
	
	\begin{defi}\thlabel{Sh.Cond}
		We denote by $\ShSp(\GgBC)$ 
		the full subcategory of  those presheaves $E$ in $\PShSp(\GgBC)$ for which the following natural morphisms are equivalences:
		\begin{enumerate}
			\item $E(X\amalg Y)\longrightarrow E(X)\oplus E(Y)$ \ for all $X,Y$ in $\GgBC$,
			\item $E(X)\longrightarrow0$ \ for all $X$ in $\GgBC$ with $X_b\simeq\emptyset$.
		\end{enumerate}
		Likewise we denote by $\ShSp(\GBC)$ the full subcategory of those presheaves $E$ in $\PShSp(\GBC)$ for which all morphisms in (1) and (2) are equivalences.\\
		We denote by $\cL$ and $\tilde\cL$ the respective localization functors 
		\begin{align*}
			\cL\colon\PShSp(\GBC)&\longrightarrow\ShSp(\GBC)\\
			\tilde\cL\colon\PShSp(\GgBC)&\longrightarrow\ShSp(\GgBC).
		\end{align*}
	\end{defi}

	\begin{thm}
		The inclusion $\iota\colon\GBC\to\GgBC$ induces an equivalence of $\infty$-categories \[
			\iota^*\colon\ShSp(\GgBC)\overset\simeq\longrightarrow\ShSp(\GBC).
		\]
	\end{thm}
	\begin{proof}
		Consider the following diagram
		\[
			\xymatrix{\GBC\ar[d]_\iota\ar[r]^(.4){\mathrm{yo}} & \PShSp(\GBC)\ar@<-3mm>[d]_{\iota_!}\ar@{}[d]|\dashv \ar@<1.8mm>[r]^\cL \ar@{}[r]|(.525){\scriptscriptstyle\bot} & 
				\ShSp(\GBC)\ar@<-3mm>@{.>}[d]_{\cF}\ar@<1.8mm>[l]^(.475)j\\
			\GgBC \ar[r]_(.4){\mathrm{yo}} & \PShSp(\GgBC)\ar@<-3mm>[u]_{\iota^*}\ar[r]_{\tilde\cL} & \ShSp(\GgBC)\ar@<-3mm>@{.>}[u]_{\iota^*}.}
		\]
		It is clear, that $\iota^*$ restricts to a functor $\ShSp(\GgBC)\to
		\ShSp(\GBC)$. Further we define $\cF$ as the composition
		$\cF:=\tilde{\cL}\circ\iota_!\circ j$, where $j$ is the inclusion
		of the full subcategory $\ShSp(\GBC)\hookrightarrow\PShSp(\GBC)$.\\[.5cm]
		Claim 1: $\iota^*$ detects equivalences. \\In fact, let $E\to F$ be
		a morphism in $\ShSp(\GgBC)$ such that $\iota^*E\to\iota^*F$ is an 
		equivalence. Then for all $X$ in $\GgBC$, \thref{iso.union} provides an isomorphism
		$X\cong X_b\amalg X_h$. By the $\tau$-sheaf conditions, claim 1 follows as:
		\begin{align*}
			E(X)&\overset\simeq\to E(X_b)\oplus E(X_h)\overset\simeq\to E(X_b)
			\simeq E(\iota(X_b))\simeq \iota^*E(X_b)\overset\simeq\to\iota^*F(X_b)\\
			&\simeq F(X_b)\simeq F(X_b\amalg X_h)
			\simeq F(X).
		\end{align*}\mbox{}\\[-.1cm]
		Claim 2: $\cL\iota^*\simeq\iota^*\tilde\cL$. We show this claim later.\\[.5cm]
		Claim 3: $\iota^*\cF\simeq\mathrm{id}$.\\
		In fact, since $\iota$ is fully faithful, the left Kan extension $\iota_!$
		is also fully faithful, hence $\iota^*\iota_!\simeq\mathrm{id}$.
		Using this and claim 2 we obtain:
		\[\iota^*\cF\simeq \iota^*\tilde\cL\iota_!j\simeq\cL\iota^*\iota_!j\simeq\cL j
		\simeq\mathrm{id}.
		\]\mbox{}\\
		Claim 4: $\cF\iota^*\simeq\mathrm{id}$.\\
		By claim 3, we have $\iota^*\cF\simeq\mathrm{id}$. In particular we get $\iota^*\cF\iota^*\simeq\iota^*$. But   $\iota^*$ detects equivalences by claim 1, thus
		 $\cF\iota^*\simeq\mathrm{id}$.\\[.5cm]
		It therefore remains to verify claim 2. The strategy will be to define two concrete models for the localization functors $\cL$ and $\tilde\cL$ and verify claim 2 by an explicit calculation.\\
		We start with constructing a model for $\tilde\cL$:\\
		To shorten notation, let $\cC:=\GBC$ and $\cD:=\GgBC$.
		We define a category $\cE$ with objects $(X,\sS)$, where $X$ is in $\cD$ 
		and $\sS$ is a decomposition of $X$ into finitely many coarse components, which is fine
			enough that each component is contained in $X_b$ or in $X_h$ (cf. \thref{thick.coarse}). 
			A morphism $(X,\sS)\to(X',\sS')$ is a morphism $f$ in $\mathrm{Mor}_\cD(X,X')$ such that $f^*\sS'\subseteq\sS$.
			Moreover, we define another category $\cE'$ with objects $(X,\sS,S)$, where $(X,\sS)$ is an object in $\cE$ and $S$ is in $\sS$ with $S\subseteq X_b$.\\
			We have forgetful functors 
			\begin{align*}
				\alpha\colon\cE'&\longrightarrow\cE,\qquad (X,\sS,S)\longmapsto(X,\sS)\\
				\beta\colon\cE&\longrightarrow\cD,\qquad \phantom{\mbox{$,S$}}(X,\sS)\longmapsto X.
			\end{align*}
			These forgetful functors induce restrictions
			\[
				\alpha^*\colon\PShSp(\cE)\to\PShSp(\cE'),\quad \beta^*\colon\PShSp(\cD)\to\PShSp(\cE).
			\]
			We denote by $\alpha_*$ the right Kan extension of $\alpha^*$ and by $\beta_!$ the left Kan extension of $\beta^*$. Finally we define a functor
			$F\colon\PShSp(\cD)\to\PShSp(\cE')$ by
			\[F(E)\big((X,\sS,S)\big):=E(S).\]
			For $(X,\sS,S)$ in $\cE'$ the inclusion $S\hookrightarrow X$ gives a natural transformation $\alpha^*\beta^*\to F$, which corresponds via the adjunction
			$\alpha^*\dashv\alpha_*$ to a natural transformation $\beta^*\to\alpha_*F$.\\
			Applying $\beta_!$ to this natural transformation and using $\beta_!\beta^*\simeq\mathrm{id}$\footnote{
			$\beta_!\beta^*(E)(X)=\colim_{\beta(X,\sS)\to X}\beta^*(E)((X,\sS))=\colim E(X)\simeq E(X)$, because $(X,\sS)$ is cofinal in the collection $\beta(Y,\sS)\to X$
			}, we obtain a natural transformation
			\[\mathrm{id}\to\beta_!\alpha_*F=:\tilde{\mathscr L}.\]
			We claim, that $\tilde{\mathscr L}$ models the localization $\tilde\cL$.\\
			To show this, we must check that $\tilde{\mathscr L}$ has values in $\ShSp(\cD)$ and that it is 
			idempotent or -- equivalently -- that $E\to\tilde{\mathscr L}E$ is an equivalence for all $\tau$-sheaves $E$. First, consider a presheaf $E$ in
			$\PShSp(\cD)$ and let $X,Y$ in $\cD$. Then for $Z:=X\amalg Y$ we calculate
			\[
				\tilde{\mathscr L}(E)(Z)=\colim_{\beta(Z,\sS)\to Z}\alpha_*F_1(E)((Z,\sS))=\colim_{\beta(Z,\sS)\to Z}\ \lim_{(Z,\sS)\to\alpha(Z,\sS,S)}\underbrace{F(E)(Z,\sS,S)}_{=E(S)}.
			\]
			We can choose $\sS$ fine enough 
			such that each $S$ belongs either to $X$ or to $Y$. Furthermore, $\sS$ was finite, hence the limit turns into a product. Hence we
			arrive at 
			\begin{align*}
				\tilde{\mathscr L}(E)(Z)&\simeq \colim_{\beta(Z,\sS)\to Z}\left(\prod_{\substack{S\in\sS\\S\subseteq X}}E(S)\times\prod_{\substack{S\in\sS\\S\subseteq Y}}E(S)\right)\\&\simeq
				\colim_{\beta(Z,\sS)\to Z}\prod_{\substack{S\in\sS\\S\subseteq X}}E(S)\ \times \colim_{\beta(Z,\sS)\to Z}\prod_{\substack{S\in\sS\\S\subseteq Y}}E(S)\\
				&\simeq\tilde{\mathscr L}(E)(X)\times\tilde{\mathscr L}(E)(Y).
			\end{align*}
			For the second equivalence we used that the colimit is filtered, hence commutes with finite products.\\
			The second sheaf condition is immediate: By definition  we require for any $S$ in $\sS$ 
			to be a subset of $X_b$. Therefore, if $X_b=\leer$, we take the colimit over the empty limit and thus
			obtain $0$. So we see, that $\tilde{\mathscr L}(E)$ is indeed
			a sheaf in $\ShSp(\cD)$.\\[.3cm]
			Next consider a sheaf $E$ in $\ShSp(\cD)$ and calculate
			\[
				\tilde{\mathscr L}(E)(X)\simeq\colim_{(X,\sS)}\lim_{(X,\sS,S)}E(S)\simeq\colim_{(X,\sS)}\prod_{S\in\sS}E(S)\simeq\colim_{(X,\sS)}E(X)\simeq E(X),
			\]
			where the third equivalence is the sheaf condition on $E$. This shows, that $\tilde{\mathscr L}$ is a model for $\tilde\cL$.\\[.5cm]
			As a second step we define $\mathscr L:=\iota^*\tilde{\mathscr L}\iota_!\colon\PShSp(\cC)\to\ShSp(\cC)$. We claim that $\mathscr L$ models
			the localization $\cL$ and further we claim, that localization commutes with
			restriction $\iota^*$, more precisely: 
			$\mathscr L\iota^*\simeq\iota^*\tilde{\mathscr L}$, which would complete the proof of claim 2.\\[.3cm]
			The second assertion follows from explicit calculation: Let $E$ be in $\PShSp(\cD)$ and let $X$ in $\cC$, then 
			\begin{align*}
				\mathscr L\iota^*(E)(X)&\simeq\tilde{\mathscr L}\iota_!\iota^*(E)(\iota(X))\simeq\tilde{\mathscr L}\iota_!\iota^*(E)(X)\\
				&\simeq\colim_{\beta(X,\sS)\to X}\ \lim_{(X,\sS)\to\alpha(X,\sS,S)}\ \colim_{\iota(Y)\to S}\iota^*(E)(Y)\\
				&\simeq \colim_{\beta(X,\sS)\to X}\ \lim_{(X,\sS)\to\alpha(X,\sS,S)} E(S) \\& \simeq\tilde{\mathscr L}(E)(X)=\iota^*\tilde{\mathscr L}(E)(X).
			\end{align*}
			For the equivalence in the third line we used that any $S$ in $\sS$ is contained in $X_b$, hence is cofinal in the collection $\iota(Y)\to S$.\\
			It remains to show, that $\mathscr L$ models the localization $\cL$, i.e. we have to check that it is idempotent. We already know that 
			$\iota^*\iota_!\simeq\mathrm{id}$ and that $\tilde{\mathscr L}$ is idempotent, hence
			\[
				\mathscr L\mathscr L\simeq \mathscr L\mathscr L\iota^*\iota_!\simeq \iota^*\tilde{\mathscr L}
				\tilde{\mathscr L}\iota_!\simeq\iota^*\tilde{\mathscr L}\iota_!\simeq \mathscr L\iota^*\iota_!\simeq \mathscr L.
			\]
	\end{proof}

To arrive at the main theorem of this section, we localize both sheaf categories in the theorem above
at the collection of morphisms below and then
compare the universal properties.

Let $\cC$ be one of the categories $\GBC$ or $\GgBC$. 

\begin{defi}\thlabel{mot.attr}
	Let $\ShSpMot(\cC)$ denote the full subcategory of those sheaves 
	$E$ in $\ShSp(\cC)$ such that the following induced morphisms are
	equivalences: 
	\begin{enumerate}
		\item $E(X)\to0$ \quad for all flasque spaces $X$ in $\cC$.
		\item $E(X,\cY)\to E(Z,Z\cap\cY)$ \quad for all 
		spaces $X$ in $\cC$ and all $\Gamma$-equivariant complementary pairs
		$(Z,\cY)$ on $X$.
		\item $E(X)\to E(\mathbb{I}\otimes X)$ \quad for all $X$ in
		$\cC$ (where $\mathbb{I}=\{0,1\}_{\max,\max}$)
		\item $E(X)\to\lim_{U\in\cC_X}E(X_U)$ \quad for all spaces $X$ in
		$\cC$.
	\end{enumerate}
\end{defi}

\begin{thm}
	We have an equivalence of categories $\ShSpMot(\GBC)\simeq\GSp$ and
	$\ShSpMot(\GgBC)\simeq\GgSp$. In particular we obtain an equivalence
	of categories $\GgSp\overset\simeq\to\GSp$.
\end{thm}
\begin{proof}
	For a category $\cC$ and a collection $\cC_0$ of objects in $\cC$ we denote by $\langle\cC_0\rangle$ the full subcategory spanned by $\cC_0$. Further we let
	$\cG$ denote one of the categories $\GBC$ or $\GgBC$ and accordingly let $\hat\cG$ denote $\GSp$ or $\GgSp$.\\
	For any stable presentable $\infty$-category $\cD$ we have the following chain of equivalences: 
	\begin{align*}
		\Fun^{\colim}(\ShSpMot(\cG),\cD)&\simeq \left\langle\left\{F\in\Fun^{\colim}(\PShSp(\cG),\cD)\ \middle|\ 
		\begin{matrix}F\text{ \footnotesize is compatible with the}\\\text{\footnotesize equivalences in \ref{mot.attr} and \ref{Sh.Cond}}\end{matrix}\right\}
		\right\rangle\\
		&\simeq \left\langle\left\{F\in\Fun^{\colim}(\Sp(\PSh(\cG)),\cD)\ \middle|\ 
		\begin{matrix}F\text{ \footnotesize is compatible with the}\\\text{\footnotesize equivalences in \ref{mot.attr} and \ref{Sh.Cond}}\end{matrix}\right\}
		\right\rangle\\
		&\simeq \left\langle\left\{F\in\Fun^{\colim}(\PSh(\cG),\cD)\ \middle|\ 
		\begin{matrix}F\text{ \footnotesize is compatible with the}\\\text{\footnotesize equivalences in \ref{mot.attr} and \ref{Sh.Cond}}\end{matrix}\right\}
		\right\rangle\\
		&\simeq\left\langle\left\{F\in\Fun(\cG,\cD)\ \middle|\ F\text{ fullfills properties in \ref{def.class.h.t}}\right\}\right\rangle\\
		&\simeq\Fun^{\colim}(\hat\cG,\cD).
	\end{align*}
	Here the first equivalence is due to \cite[Prop 5.5.4.20]{HTT}, the second follows from $\Fun(\cC,\Sp(\Spc))\simeq\Sp(\Fun(\cC,\Spc))$ by \cite[Rmk 1.4.2.9]{HA} and the
	third is due to \cite[Cor 1.4.4.5]{HA}. The last equivalence follows from \thref{UE.1,Sigma.adj} in the generalized case and from \cite[Cor 4.5]{Uli} in the classical case. 
	Therefore it remains to justify the forth equivalence:\\
	First we notice that compatibility with the equivalences in \thref{Sh.Cond} is already implied by the compatibility with those equivalences in \thref{mot.attr} (using
	\thref{huge.is.flasque,Xadd}). Now the equivalence follows from \cite[Thm 5.1.5.6]{HTT}.
\end{proof}

Let $\cC$ be a co-complete stable $\infty$-category. 
For the purpose of the following corollary we denote by $\boldsymbol{\mathcal X}\mathbf{-HomolTheo}(\cC)$ the full subcategory of $\Fun(\GBC,\cC)$ spanned by
coarse homology theories (cf. \cite[Def. 3.10]{equ.Uli}). Likewise we denote by $\boldsymbol{\widetilde{\mathcal X}}\mathbf{-}{\mathbf{HomolTheo}}$ the full subcategory
of $\Fun(\GgBC,\cC)$ spanned by (generalized) coarse homology theories (cf. \thref{def.class.h.t}).

\begin{cor}
	The inclusion functor $\GBC\to\GgBC$ induces an equivalence of categories
	\[
		\boldsymbol{\widetilde{\mathcal X}}\mathbf{-}{\mathbf{HomolTheo}} \overset\simeq\longrightarrow\boldsymbol{\mathcal X}\mathbf{-HomolTheo}(\cC).
	\]
\end{cor}

\vfill

\section*{Acknowledgments} \addtocontents{section}{Acknowledgments}
	The author deeply thanks his adviser Ulrich Bunke for providing the subject and his guidance and input 
	through writing this paper. Further the author is indebted to Markus Land for helpful comments and 
	discussions and his interest in the topic.
	
	\vspace{2cm}

	\printbibliography
\end{document}